\documentclass{article}
\usepackage{PRIMEarxiv}
\usepackage[T1]{fontenc}
\usepackage[utf8]{inputenc}
\usepackage{lmodern}
\usepackage{amsmath,amssymb,amsthm,mathtools}
\usepackage{microtype}
\usepackage{enumitem}
\usepackage{xcolor}
\usepackage{hyperref}
\usepackage{tikz}
\usepackage{pgfplots}
\usepackage{tikz-3dplot}
\usetikzlibrary{arrows.meta,calc,patterns,positioning}
\pgfplotsset{compat=1.18}
\hypersetup{colorlinks=true,linkcolor=blue!50!black,urlcolor=blue!50!black,citecolor=blue!50!black}
\definecolor{hvblue}{RGB}{36,102,204}
\definecolor{magred}{RGB}{196,51,58}
\definecolor{softgray}{RGB}{245,246,248}

\newtheorem{theorem}{Theorem}[section]
\newtheorem{proposition}[theorem]{Proposition}
\newtheorem{corollary}[theorem]{Corollary}
\newtheorem{lemma}[theorem]{Lemma}
\newtheorem{conjecture}[theorem]{Conjecture}
\theoremstyle{definition}
\newtheorem{definition}[theorem]{Definition}
\newtheorem{remark}[theorem]{Remark}
\newtheorem{example}[theorem]{Example}

\newcommand{\Mag}{\operatorname{Mag}}
\newcommand{\HV}{\operatorname{HV}}
\newcommand{\R}{\mathbb{R}}
\newcommand{\1}{\mathbf{1}}
\newcommand{\proj}{\pi}
\newcommand{\vol}{\operatorname{vol}}

\newcommand{\myparagraph}[1]{\medskip\noindent\textbf{#1}}

\title{The Magnitude of Dominated Sets: A Pareto Compliant Indicator Grounded in Metric Geometry}
\author{Michael T.M. Emmerich\\Faculty of Information Technology, University of Jyv\"askyl\"a, Finland}
\date{}

\begin{document}
\maketitle

\begin{abstract}
We investigate \emph{magnitude} as a new unary and strictly Pareto-compliant quality indicator for finite approximation sets to the Pareto front in multiobjective optimization. Magnitude originates in enriched category theory and metric geometry, where it is a notion of size or point content for compact metric spaces and a generalization of cardinality. For dominated regions in the \(\ell_1\) box setting, magnitude is close to hypervolume but not identical: it contains the top-dimensional hypervolume term together with positive lower-dimensional projection and boundary contributions.

This paper gives a first theoretical study of magnitude as an indicator. We consider multiobjective maximization with a common anchor point. For dominated sets generated by finite approximation sets, we derive an all-dimensional projection formula, prove weak and strict set monotonicity on finite unions of anchored boxes, and thereby obtain weak and strict Pareto compliance. Unlike hypervolume, magnitude assigns positive value to boundary points sharing one or more coordinates with the anchor point, even when their top-dimensional hypervolume contribution vanishes. We then formulate projected set-gradient methods and compare hypervolume and magnitude on biobjective and three-dimensional simplex examples. Numerically, magnitude favors boundary-including populations and, for suitable cardinalities, complete Das-Dennis grids, whereas hypervolume prefers more interior-filling configurations. Computationally, magnitude reduces to hypervolume on coordinate projections; for fixed dimension this yields the same asymptotic complexity up to a factor \(2^d-1\), and in dimensions two and three \(\Theta(n\log n)\) time. These results identify magnitude as a mathematically natural and computationally viable alternative to hypervolume for finite Pareto front approximations.

\end{abstract}

\section{Introduction}
The approximation of Pareto fronts by finite sets is a central theme in multiobjective optimization. A natural way to compare such approximations is to assign to each finite set a numerical value that measures the size or quality of the dominated region induced by that set. This leads to the study of \emph{set indicators}, among which the hypervolume indicator is the best-known example; see Miettinen~\cite{Miettinen1999} for a general introduction to nonlinear multiobjective optimization, Zitzler et al.~\cite{Zitzler2003} for the comparison and assessment of Pareto front approximations, and Emmerich and Deutz~\cite{EmmerichDeutz2018} for a tutorial on the finite-set approximation viewpoint and the basic indicator concepts used in this paper. In this paper we investigate a different indicator, namely the \emph{magnitude of the dominated set}, motivated by the notion of magnitude from enriched categories and compact metric spaces, where it appears as a generalized notion of cardinality or point content. Our perspective is theoretical: we study the structural properties of this indicator for finite Pareto front approximations in multiobjective maximization, with particular emphasis on set monotonicity, Pareto compliance, and the geometry of indicator-optimal point distributions. Thus, the paper is situated at the interface of indicator theory, evolutionary and set-oriented multiobjective optimization, and geometric notions of size. Hypervolume-based archiving and heuristic selection have also played an important role in indicator-based evolutionary search~\cite{KnowlesCorneFleischer2003,Fleischer2003}.

Throughout the paper, objective vectors are viewed as elements of subsets of \(\R^d\). A finite \emph{approximation set} is a finite set \(A\subseteq \R^d\) intended to approximate the Pareto front in objective space. Since we consider maximization problems, an objective vector \(y\in\R^d\) \emph{weakly dominates} \(z\in\R^d\) if \(y_i\ge z_i\) for all \(i=1,\dots,d\), and \emph{strictly dominates} it if, in addition, \(y_j>z_j\) for at least one index \(j\). The \emph{Pareto front} is the set of nondominated attainable objective vectors. Whenever an objective map is introduced, we write it as \(F=(F_1,\dots,F_d):\Omega	\rightarrow \R^d\), so that uppercase \(F_i\) always denote objective functions, while lowercase symbols are reserved for points or coordinates in objective space.

Magnitude was introduced in enriched category theory and has developed into a geometric notion of size for compact metric spaces~\cite{Leinster2013,LeinsterMeckes2017,LeinsterTalk}. Depending on the setting, it behaves like effective number of points, point content, or a cardinality-like invariant. For compact subsets of $\ell_1^n$, and in particular for unions of axis-parallel boxes, magnitude has a remarkable structure: it combines the top-dimensional volume term with positive lower-dimensional boundary, projection, and shadow terms. This immediately suggests a comparison with hypervolume in multiobjective optimization. Hypervolume measures only the dominated volume, whereas magnitude measures the size of the same dominated region in a richer geometric sense.

The central motivation of this paper is therefore fundamental: \emph{if a finite Pareto approximation is viewed as generating a dominated set, then magnitude provides another principled way of quantifying the size of that set}. Because magnitude generalizes cardinality in metric spaces while also incorporating geometric content, it is a natural candidate for a new indicator that is similar to hypervolume but not the same. The present paper is a first theoretical investigation of that idea, with special emphasis on dominated-set structure, monotonicity, and Pareto compliance.

Our study focuses on dominated regions of the form $D(A)=\bigcup_{a\in A}[0,a]$ and on low-dimensional box geometries where exact analysis is still possible. In two dimensions we first work with zero-anchored unions of rectangles and show how magnitude decomposes into hypervolume plus positive lower-dimensional terms. We then extend the discussion to dominated sets in arbitrary dimension by expressing magnitude as a positive weighted sum of measures of coordinate projections. This yields strict set monotonicity on the class of finite unions of anchored boxes and therefore Pareto compliance for dominated-set comparisons. It also explains why magnitude should put more emphasis on extreme and boundary points than hypervolume: points that share coordinates with the anchor point still contribute to magnitude through lower-dimensional terms, whereas their top-dimensional hypervolume contribution vanishes. We then derive set-gradient methods for both indicators and compare their behavior on two biobjective examples.

\begin{example}[A three-point dominated set in two objectives]\label{ex:intro-three-point-mag}
Consider the three objective vectors
\[
(1,3),\qquad (3,2),\qquad (5,1)
\]
in a biobjective \emph{maximization} problem with anchor point $(0,0)$. The corresponding dominated set is the union of the three anchored rectangles
\[
[0,1]\times[0,3],\qquad [0,3]\times[0,2],\qquad [0,5]\times[0,1].
\]
Figure~\ref{fig:intro-three-point-mag} shows the resulting dominated region. Since the union is zero-anchored and planar, Theorem~\ref{thm:planar-dominated-magnitude} below yields
\[
\Mag(D)=1+\frac{X+Y}{2}+\frac{\HV(D)}{4}.
\]
Here $X=5$, $Y=3$, and the dominated area is
\[
\HV(D)=3\cdot 1+2\cdot(3-1)+1\cdot(5-3)=9.
\]
Therefore
\[
\Mag(D)=1+\frac{5+3}{2}+\frac{9}{4}=\frac{29}{4}.
\]
This simple example already illustrates the central difference from hypervolume: the magnitude contains the top-dimensional dominated area, but also positive lower-dimensional contributions coming from the coordinate shadows of the dominated set.

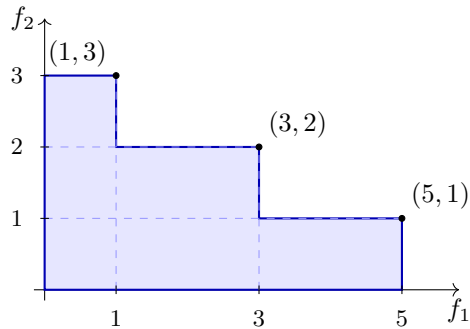
\begin{figure}[t]
\centering
\begin{tikzpicture}[x=0.95cm,y=0.95cm]
  % axes
  \draw[->] (-0.15,0) -- (5.8,0) node[below] {$f_1$};
  \draw[->] (0,-0.15) -- (0,3.8) node[left] {$f_2$};

  % dominated union fills
  \fill[blue!10] (0,0) rectangle (1,3);
  \fill[blue!10] (0,0) rectangle (3,2);
  \fill[blue!10] (0,0) rectangle (5,1);

  % boundary of union
  \draw[thick,blue!70!black] (0,0) -- (5,0) -- (5,1) -- (3,1) -- (3,2) -- (1,2) -- (1,3) -- (0,3) -- cycle;

  % individual rectangle hints
  \draw[blue!40,dashed] (1,0) -- (1,3);
  \draw[blue!40,dashed] (3,0) -- (3,2);
  \draw[blue!40,dashed] (0,1) -- (5,1);
  \draw[blue!40,dashed] (0,2) -- (3,2);

  % points
  \filldraw[black] (1,3) circle (1.1pt) node[above left] {$(1,3)$};
  \filldraw[black] (3,2) circle (1.1pt) node[above right] {$(3,2)$};
  \filldraw[black] (5,1) circle (1.1pt) node[above right] {$(5,1)$};

  % ticks
  \foreach \x in {1,3,5} \draw (\x,0.06) -- (\x,-0.06) node[below=3pt] {\small \x};
  \foreach \y in {1,2,3} \draw (0.06,\y) -- (-0.06,\y) node[left=3pt] {\small \y};
\end{tikzpicture}
\caption{Dominated set generated by an approximation set $A=\{(1,3), (3,2), (5,1)\}$ consisting of three objective vectors  with anchor point $(0,0)$. The dominated rectangles are zero-anchored, and their union has hypervolume $9$ and magnitude $29/4$.}
\label{fig:intro-three-point-mag}
\end{figure}
\end{example}

The three-dimensional simplex case sharpens the contrast further. It also places the present study next to existing work on optimal indicator distributions and analytical hypervolume calculations on linear and quadratic fronts~\cite{Brockhoff2010,IshibuchiImadaMasuyamaNojima2019,Singh2025,IshibuchiNanPang2025}. There, exact inclusion--exclusion formulas remain tractable, allowing us to compare hypervolume and magnitude on symmetric six-point populations and on complete Das--Dennis grids, i.e.\ simplex-lattice parameter sets in the sense of Das and Dennis~\cite{DasDennis1998}. The numerical evidence indicates that magnitude tends to converge to populations that include the boundary and, for appropriate cardinalities, to regular Das--Dennis grids, while hypervolume keeps rewarding more interior-filling layouts. This leads to an intriguing stationarity phenomenon for magnitude on the simplex and suggests new connections between indicator theory, dominated-set geometry, and reference-direction constructions.

Throughout this paper we consider multiobjective \emph{maximization} problems. We use the term \emph{anchor point} rather than \emph{reference point}. Our canonical choice is the origin $0=(0,\dots,0)$, which is assumed to lie below the objective vectors under consideration and may be interpreted as a nadir-type anchor point. This normalization entails no loss of generality, since any other common anchor point can be shifted to the origin by coordinate-wise translation of the dominated set.

\paragraph{Contributions.}
The main contributions of this paper are as follows.
\begin{enumerate}[leftmargin=2em]
\item We present a self-contained low-dimensional theory of magnitude for dominated regions in the $\ell_1$ box setting and make its relation to hypervolume explicit.
\item We prove an all-dimensional projection formula, strict set monotonicity for finite unions of anchored boxes, and therefore Pareto compliance for magnitude on dominated sets.
\item We derive projected set-gradient methods for hypervolume and magnitude, using normalized per-point gradients in the numerical runs.
\item We show numerically that magnitude differs systematically from hypervolume by favoring boundary-including populations; in three dimensions, magnitude aligns naturally with complete Das--Dennis grids for suitable population sizes.
\item We derive exact inclusion--exclusion formulas for the three-dimensional simplex setting, study exact small-cardinality configurations, and sketch the route to higher-dimensional magnitude.
\end{enumerate}

The remainder of the paper is organized as follows. Section~2 recalls the magnitude background needed in the $\ell_1$ setting. Section~3 develops the low-dimensional formulas and their relation to hypervolume. Section~4 establishes weak and strict set monotonicity and derives weak and strict Pareto compliance on dominated sets. Section~5 introduces the projected set-gradient framework. Sections~6--9 treat the two biobjective examples and the corresponding exact indicator-optimal distributions. Sections~10--12 turn to the three-dimensional simplex setting and analyze the exact inclusion--exclusion formulas, small-cardinality configurations, and the Das--Dennis stationarity phenomenon. Section~13 discusses projection-based computation of magnitude. Section~14 discusses possible advantages and disadvantages of magnitude as an indicator. We close in Section~15 with conclusions and future directions.

\section{Preliminaries on magnitude in \texorpdfstring{$\ell_1$}{l1}}

We begin by recalling the basic facts about magnitude that will be used throughout the paper. Since the notion is less familiar in optimization than hypervolume, we start with some intuition. Magnitude is a notion of the \emph{size} of a metric space. For finite spaces, it is built from the matrix
\[
\bigl(e^{-d(a_i,a_j)}\bigr)_{i,j},
\]
so the interaction between two points decreases exponentially with their distance. Heuristically, the exponential is natural here because it is the basic function that turns sums of distances into products:
\[
e^{-(x+y)}=e^{-x}e^{-y}.
\]
More generally, among continuous functions \(f\) with \(f(0)=1\) and \(f(x+y)=f(x)f(y)\), the exponential is the canonical choice. Thus the entries \(e^{-d(a_i,a_j)}\) encode how strongly different points ``overlap'' as parts of the same metric object. Closely spaced points interact strongly, while far apart points behave more like separate contributions to size.

Our focus is on the \(\ell_1\) setting. Here \(\ell_1\) means that for points
\[
x=(x_1,\dots,x_n),\qquad y=(y_1,\dots,y_n)\in \R^n,
\]
the distance is
\[
d_{\ell_1}(x,y)=\sum_{i=1}^n |x_i-y_i|.
\]
This metric is especially well suited to the present paper because the dominated regions relevant to multiobjective optimization are unions of axis-aligned boxes, and in \(\ell_1\) such boxes admit particularly explicit formulas. In particular, products behave cleanly, and low-dimensional magnitude formulas become directly comparable to hypervolume plus lower-dimensional shadow terms. This makes \(\ell_1\) much more tractable here than a general compact metric-space setting.

\myparagraph{Finite metric spaces and compact sets}
\begin{definition}
Let \(A=\{a_1,\dots,a_m\}\) be a finite metric space. Its similarity matrix is
\[
Z_A=(e^{-d(a_i,a_j)})_{i,j=1}^m.
\]
If \(Z_A\) is invertible, the \emph{magnitude} of \(A\) is
\[
\Mag(A)=\sum_{i,j=1}^m (Z_A^{-1})_{ij}.
\]
Equivalently, if \(w\in\R^m\) satisfies \(Z_A w=\1\), then
\[
\Mag(A)=\sum_{i=1}^m w_i.
\]
\end{definition}

The second formulation is often useful heuristically: the entries of \(w\) may be viewed as weights assigned to the points so that every point sees total weighted similarity exactly equal to \(1\). The magnitude is then the total weight of the space. In this sense, magnitude behaves like a generalized effective cardinality.

For compact positive-definite metric spaces, including compact subsets of \(\ell_1^n\), magnitude is extended by approximation by finite subsets. We keep the same notation \(\Mag(A)\) for the compact case.

\myparagraph{Intervals and products}

The first examples already show why \(\ell_1\) is so convenient. On an interval, magnitude is linear in the length, and on Cartesian products it is multiplicative. These two facts together will later give explicit formulas for rectangles and boxes.

\begin{proposition}[interval formula]
For an interval \([0,L]\subset \R\) with the induced \(\ell_1\) metric,
\[
\Mag([0,L])=1+\frac{L}{2}.
\]
\end{proposition}

Thus even in one dimension, magnitude is not just length: it is a size functional with a constant term and a geometric term.

\begin{proposition}[product formula]
If \(A\subseteq \R^m\) and \(B\subseteq \R^n\) are compact and carry the induced \(\ell_1\) metrics, then
\[
\Mag(A\times B)=\Mag(A)\Mag(B),
\]
where \(A\times B\subseteq \R^{m+n}\) is equipped with the product \(\ell_1\) metric.
\end{proposition}

\begin{proof}
For finite spaces the similarity matrix of \(A\times B\) is the Kronecker product of the similarity matrices of \(A\) and \(B\). Summing the entries of the inverse therefore multiplies. The compact case follows by approximation.
\end{proof}

\myparagraph{Rectangles, boxes, and zero-anchored unions}

By combining the interval formula with multiplicativity, one immediately obtains explicit formulas for rectangles and boxes. These are exactly the shapes that arise in dominated regions.

\begin{proposition}[single rectangle and single box]
For a rectangle \(R=[0,a]\times[0,b]\subset \ell_1^2\) and a box \(Q=[0,a]\times[0,b]\times[0,c]\subset \ell_1^3\),
\[
\Mag(R)=\Bigl(1+\frac a2\Bigr)\Bigl(1+\frac b2\Bigr)=1+\frac{a+b}{2}+\frac{ab}{4},
\]
\[
\Mag(Q)=\Bigl(1+\frac a2\Bigr)\Bigl(1+\frac b2\Bigr)\Bigl(1+\frac c2\Bigr)
=1+\frac{a+b+c}{2}+\frac{ab+ac+bc}{4}+\frac{abc}{8}.
\]
\end{proposition}

\begin{proof}
Apply the interval formula coordinatewise and use multiplicativity.
\end{proof}

These formulas already suggest the main point: in \(\ell_1\), magnitude decomposes into a top-dimensional volume term together with lower-dimensional terms. For a rectangle, the area term \(ab/4\) is accompanied by the boundary-shadow term \((a+b)/2\); for a box, the volume term \(abc/8\) is accompanied by face and edge contributions. This is exactly the mechanism that later distinguishes magnitude from hypervolume.

The class most relevant to hypervolume is the family of downward-closed, zero-anchored planar unions
\[
U=\bigcup_{k=1}^m [0,x_k]\times[0,y_k].
\]
For the \(\ell_1\)-convex pixelated sets treated in Leinster and Meckes' \(\ell_1\) theory \cite{LeinsterMeckes2017}, one has the exact low-dimensional valuation formula
\[
\Mag(U)=V_0'(U)+\frac{V_1'(U)}{2}+\frac{V_2'(U)}{4}.
\]
For zero-anchored unions the terms become especially concrete: the zeroth-order term is \(1\), the first-order term is the sum of the coordinate projection lengths, and the second-order term is the ordinary area.

\begin{theorem}[zero-anchored planar union formula]\label{thm:planar-dominated-magnitude}
Let
\[
U=\bigcup_{k=1}^m [0,x_k]\times[0,y_k]\subset \ell_1^2,
\qquad x_k,y_k\ge 0,
\]
and assume \(U\) belongs to the exact class above. Then
\[
\Mag(U)=1+\frac{X+Y}{2}+\frac{\HV(U)}{4},
\]
where
\[
X=\max_k x_k,
\qquad
Y=\max_k y_k,
\qquad
\HV(U)=\vol_2(U).
\]
\end{theorem}

\begin{proof}
The set \(U\) is contractible, so the Euler-characteristic term is \(1\). Its coordinate projections are \([0,X]\) and \([0,Y]\), so the first intrinsic-volume term is \(X+Y\). The top term is ordinary area, that is, the hypervolume of the dominated region.
\end{proof}

This formula is one of the key reasons why magnitude is so attractive in the present setting: in two dimensions it is literally hypervolume plus a positive correction coming from the two one-dimensional shadows of the dominated set.

\begin{remark}
Example~\ref{ex:intro-three-point-mag} in the introduction gives a concrete exact two-dimensional magnitude calculation of this type. The same pattern persists for a single \(n\)-box \(B=\prod_{i=1}^n [0,L_i]\):
\[
\Mag(B)=\prod_{i=1}^n\Bigl(1+\frac{L_i}{2}\Bigr)=\sum_{k=0}^n \frac{e_k(L_1,\dots,L_n)}{2^k},
\]
where \(e_k\) is the \(k\)th elementary symmetric polynomial. Thus magnitude of a box is built from the full-dimensional term together with all lower-dimensional coordinate-face contributions. In the present paper, we keep the higher-dimensional discussion at the level of this product formula and the low-dimensional shadow interpretation.
\end{remark}

\section{Monotonicity, Pareto compliance, and relation to hypervolume}

We now turn from explicit low-dimensional formulas to the structural consequences for indicator theory. The key observation is that dominated regions generated by approximation sets in multiobjective maximization are \emph{downward closed} and, after translation of a common anchor point to the origin, are finite unions of anchored boxes. This makes it possible to formulate magnitude in a way that is both all-dimensional and directly relevant to Pareto compliance.

\begin{definition}
For a finite approximation set $A\subseteq [0,\infty)^m$ in a maximization problem with anchor point $0$, its dominated region is
\[
D(A):=\bigcup_{a\in A}[0,a],
\qquad
[0,a]:=\prod_{i=1}^m [0,a_i].
\]
A set $D\subseteq [0,\infty)^m$ is called \emph{anchored downward closed} if $x\in D$ and $0\le y\le x$ imply $y\in D$.
\end{definition}

Every dominated region $D(A)$ is anchored downward closed and is a finite union of anchored boxes.

\begin{lemma}[projection of intersections for downward-closed sets]
Let $D,E\subseteq [0,\infty)^m$ be downward closed and let $S\subseteq [m]:=\{1,\dots,m\}$. Then
\[
\proj_S(D\cap E)=\proj_S(D)\cap \proj_S(E),
\]
where $\proj_S$ denotes coordinate projection onto the coordinates in $S$.
\end{lemma}

\begin{proof}
The inclusion $\proj_S(D\cap E)\subseteq \proj_S(D)\cap \proj_S(E)$ is immediate. For the converse, let $x_S\in \proj_S(D)\cap \proj_S(E)$. Then there exist complementary coordinate vectors $y$ and $z$ such that $(x_S,y)\in D$ and $(x_S,z)\in E$. Let $w:=\min(y,z)$ coordinatewise. Since $D$ and $E$ are downward closed, $(x_S,w)\in D\cap E$. Hence $x_S\in \proj_S(D\cap E)$.
\end{proof}

\begin{theorem}[all-dimensional magnitude formula for dominated sets]
Let $D\subseteq [0,\infty)^m$ be a finite union of anchored boxes. Then
\[
\Mag(D)=\sum_{S\subseteq [m]} 2^{-|S|}\,\lambda_{|S|}(\proj_S D),
\]
where $\lambda_k$ denotes $k$-dimensional Lebesgue measure, and by convention $\lambda_0(\proj_\emptyset D)=1$ for nonempty $D$.
\end{theorem}

\begin{proof}
Define
\[
F(D):=\sum_{S\subseteq [m]} 2^{-|S|}\,\lambda_{|S|}(\proj_S D).
\]
By the lemma, for each fixed $S$ the map $D\mapsto \lambda_{|S|}(\proj_S D)$ is a valuation on downward-closed sets, because projection sends unions to unions and, in the downward-closed class, sends intersections to intersections. Hence $F$ is itself a valuation on finite unions of anchored boxes.

Now consider one anchored box $Q(a)=[0,a_1]\times\cdots\times[0,a_m]$. Its projection onto $S$ is again an anchored box of measure $\prod_{i\in S} a_i$, so
\[
F(Q(a))=\sum_{S\subseteq [m]} 2^{-|S|}\prod_{i\in S} a_i
=\prod_{i=1}^m \left(1+\frac{a_i}{2}\right).
\]
By the product formula for magnitude in $\ell_1$, the right-hand side is exactly $\Mag(Q(a))$. Thus $F$ and $\Mag$ agree on anchored boxes, and since both are valuations on the class generated by anchored boxes, they agree on every finite union of anchored boxes.
\end{proof}

The formula specializes in dimension two to
\[
\Mag(D)=1+\frac{\lambda_1(\proj_{\{1\}}D)+\lambda_1(\proj_{\{2\}}D)}{2}+\frac{\lambda_2(D)}{4},
\]
and in dimension three to
\begin{align*}
\Mag(D)&=1+\frac{\lambda_1(\proj_{\{1\}}D)+\lambda_1(\proj_{\{2\}}D)+\lambda_1(\proj_{\{3\}}D)}{2}\\
&\quad+\frac{\lambda_2(\proj_{\{1,2\}}D)+\lambda_2(\proj_{\{1,3\}}D)+\lambda_2(\proj_{\{2,3\}}D)}{4}\\
&\quad+\frac{\lambda_3(D)}{8}.
\end{align*}
Thus hypervolume appears as the top-dimensional term, but magnitude also contains strictly positive lower-dimensional projection terms.

\begin{theorem}[Weak and strict set monotonicity of magnitude on dominated sets]
Let $D,E\subseteq [0,\infty)^m$ be finite unions of anchored boxes. If $D\subseteq E$, then
\[
\Mag(D)\le \Mag(E).
\]
If $D\subsetneq E$, then
\[
\Mag(D)<\Mag(E).
\]
Thus magnitude is weakly and strictly set-monotone on the class of anchored dominated sets.
\end{theorem}

\begin{proof}
If $D\subseteq E$, then for every $S\subseteq [m]$ one has $\proj_S D\subseteq \proj_S E$, hence
\[
\lambda_{|S|}(\proj_S D)\le \lambda_{|S|}(\proj_S E).
\]
Since all coefficients $2^{-|S|}$ in the projection formula for magnitude are positive, it follows immediately that
\[
\Mag(D)\le \Mag(E).
\]
This proves weak set monotonicity.

If the inclusion is strict, choose $x\in E\setminus D$ and let $T:=\{i:x_i>0\}$. Then $x_T\in \proj_T(E)$. We claim that $x_T\notin \proj_T(D)$. Indeed, if $x_T\in \proj_T(D)$, then there exists a complementary coordinate vector $y$ with $(x_T,y)\in D$. Since $D$ is downward closed and $x=(x_T,0)\le (x_T,y)$, this would imply $x\in D$, a contradiction. Therefore $\proj_T(D)\subsetneq \proj_T(E)$.

Because $D$ and $E$ are finite unions of anchored boxes, their coordinate projections are finite unions of anchored boxes as well, hence measurable and closed. Since $x_T\in \proj_T(E)\setminus \proj_T(D)$, there exists a small anchored neighborhood below $x_T$ that lies in $\proj_T(E)$ but not in $\proj_T(D)$. Consequently,
\[
\lambda_{|T|}(\proj_T D)<\lambda_{|T|}(\proj_T E).
\]
At least one term in the magnitude formula is therefore strictly larger for $E$, so $\Mag(D)<\Mag(E)$. This proves strict set monotonicity.
\end{proof}

\begin{remark}[Boundary monotonicity versus hypervolume]
The preceding theorem already shows a structural difference between magnitude and hypervolume. Hypervolume depends only on the top-dimensional measure $\lambda_m(D)$ and therefore does not change when a dominated set is enlarged only through lower-dimensional boundary parts with zero $m$-dimensional measure. Magnitude, by contrast, is monotone with respect to all coordinate-shadow enlargements because every projection term enters with a positive coefficient. In particular, adding boundary points with one or more anchor-point coordinates may leave hypervolume unchanged while still increasing magnitude.
\end{remark}

\myparagraph{Pareto compliance from set monotonicity}
We recall the standard indicator notion; see, for example, Zitzler, Brockhoff, and Thiele~\cite{ZitzlerBrockhoffThiele2007}.

\begin{definition}
Let $I$ be a set indicator for finite approximation sets in a maximization problem with fixed anchor point $r$. For finite approximation sets $A$ and $B$, write
\[
A \preceq B
\]
if every point of $A$ is weakly dominated by some point of $B$.
We say that $I$ is \emph{weakly Pareto compliant} if
\[
A \preceq B \quad\Longrightarrow\quad I(A)\le I(B).
\]
We say that $I$ is \emph{strictly Pareto compliant} if
\[
A \preceq B \text{ and } B \not\preceq A \quad\Longrightarrow\quad I(A)<I(B).
\]
Thus, with our convention, $A \preceq B$ means that $B$ is at least as good as $A$.
\end{definition}

For hypervolume, the implication from dominated-region inclusion to indicator monotonicity is standard and underlies its Pareto-compliance interpretation; see also Emmerich, Deutz, and Yevseyeva~\cite{EmmerichDeutzYevseyeva2014} for a related discussion of dominated-region based hypervolume constructions. Once weak and strict set monotonicity have been established, the corresponding Pareto-compliance statements are immediate.

\begin{corollary}[Weak and strict Pareto compliance of magnitude on dominated sets]
Fix an anchor point $r\in\R^m$ and let $A$ and $B$ be finite approximation sets in an $m$-objective maximization problem. Suppose that $r$ lies below the objective vectors under consideration, so that it serves as a common nadir-type anchor point, and suppose that the translated dominated regions $\widetilde D_r(A)$ and $\widetilde D_r(B)$ are finite unions of anchored boxes in $[0,\infty)^m$.

If $A\preceq B$, then
\[
\Mag(\widetilde D_r(A))\le \Mag(\widetilde D_r(B)).
\]
If, in addition, $B\not\preceq A$, then
\[
\Mag(\widetilde D_r(A))< \Mag(\widetilde D_r(B)).
\]
Hence magnitude is weakly Pareto compliant on this class. Moreover, it is strictly Pareto compliant on this class.
\end{corollary}

\begin{proof}
If $A\preceq B$, then every point dominated by $A$ is also dominated by $B$, so $\widetilde D_r(A)\subseteq \widetilde D_r(B)$. The weak inequality therefore follows from weak set monotonicity.
If, in addition, $B\not\preceq A$, then there exists a point $b\in B$ that is not weakly dominated by any point of $A$. After translation by the common anchor point, the vector $b-r$ belongs to $\widetilde D_r(B)$ but not to $\widetilde D_r(A)$. Hence $\widetilde D_r(A)\subsetneq \widetilde D_r(B)$, and the strict inequality follows from strict set monotonicity.
\end{proof}

\begin{remark}[Boundary-touching points: a key difference from hypervolume]
The corollary also makes clear why magnitude and hypervolume respond differently to points on coordinate faces through the anchor point. Hypervolume is purely top-dimensional:
\[
\HV(D)=\lambda_m(D).
\]
Therefore a point such as $(a,b,0)$ in three objectives contributes no positive three-dimensional volume increment by itself. Magnitude, however, contains all lower-dimensional projection terms as well. The same point contributes positively through the projections onto coordinate axes and coordinate planes. Thus points that share one or more coordinates with the anchor point have positive contributions to magnitude, whereas their top-dimensional hypervolume contribution vanishes. This is one of the fundamental structural differences between the two indicators and explains why magnitude naturally rewards boundary-including approximation sets.
\end{remark}

In multiobjective maximization with anchor point $r\in\R^m$, the dominated region of a finite set $P=\{p^{(1)},\dots,p^{(\mu)}\}$ is
\[
D_r(P)=\bigcup_{i=1}^{\mu} \prod_{j=1}^m [r_j,p^{(i)}_j].
\]
After translation by $r$, the same region becomes zero-anchored. Hence magnitude with respect to a common anchor point is simply the magnitude of the translated dominated region. The canonical choice $r=0$ is therefore merely a normalization.

\section{A projected set-gradient ascent method}

Having established the indicator-theoretic properties of magnitude, we next explain how to optimize it over finite populations. The goal of this section is to formulate a set-gradient method in a way that makes the comparison between hypervolume and magnitude completely parallel.

The projected set-gradient idea is standard in indicator-based search; for hypervolume-gradient and Newton-type developments see, for example, Sosa Hernández et al.~\cite{SosaHernandezSchuetzeEmmerich2014}, Emmerich and Deutz~\cite{EmmerichDeutz2014}, Hernández et al.~\cite{HernandezSchuetzeWangDeutzEmmerich2018}, and Wang et al.~\cite{WangDeutzBackEmmerich2017}. Here we formulate it in a way that makes the comparison between hypervolume and magnitude completely parallel.

Consider a smooth biobjective maximization problem
\[
\max_{z\in\Omega} F(z)=(F_1(z),F_2(z)),
\qquad \Omega\subseteq \R^d,
\]
with a fixed anchor point $r=(r_1,r_2)$ satisfying $r_j\le F_j(z)$ for all relevant objective vectors. For a population $P=\{z^{(1)},\dots,z^{(\mu)}\}$, define the translated objective vectors
\[
\widetilde y^{(i)}=F(z^{(i)})-r
\in \R_{\ge 0}^2,
\]
and the translated dominated region
\[
\widetilde D(P)=\bigcup_{i=1}^{\mu}[0,\widetilde y_1^{(i)}]\times[0,\widetilde y_2^{(i)}].
\]
We maximize either
\[
J_{\mathrm{HV}}(P)=\HV(\widetilde D(P))
\qquad\text{or}\qquad
J_{\mathrm{Mag}}(P)=\Mag(\widetilde D(P)).
\]

\myparagraph{Generic pull-back formula}
Assume the indicator $J(P)$ is differentiable with respect to the active objective vectors after sorting and removing dominated duplicates. Then the population gradient with respect to the $i$th search point is
\[
\nabla_{z^{(i)}} J(P)=D F(z^{(i)})^\top \nabla_{\widetilde y^{(i)}} J(P).
\]
In the numerical experiments below we use the \emph{normalized pull-back direction}
\[
g^{(i)}(P)=
\begin{cases}
\dfrac{D F(z^{(i)})^\top \nabla_{\widetilde y^{(i)}} J(P)}{\left\|D F(z^{(i)})^\top \nabla_{\widetilde y^{(i)}} J(P)\right\|_2}, & \text{if } D F(z^{(i)})^\top \nabla_{\widetilde y^{(i)}} J(P)\neq 0,\\[1.0ex]
0, & \text{otherwise.}
\end{cases}
\]
A projected ascent step is therefore implemented as
\[
z_{k+1}^{(i)}=\Pi_{\Omega}\Bigl(z_k^{(i)}+\eta_k \, g^{(i)}(P_k)\Bigr),
\qquad i=1,\dots,\mu.
\]
This normalization makes the update less sensitive to large indicator derivatives at a few isolated points and proved numerically more robust than the unnormalized pull-back in the experiments reported below.

\myparagraph{Practical algorithm}
\myparagraph{Projected set-gradient ascent}
\begin{enumerate}[label=\arabic*.]
    \item Start with a population $P_0=\{z_0^{(1)},\dots,z_0^{(\mu)}\}\subset\Omega$.
    \item Map the population to objective space and translate by the anchor point.
    \item Remove duplicate and dominated objective vectors.
    \item Sort the remaining active points along the front.
    \item Compute the indicator derivative in objective space.
    \item Pull it back by the Jacobian $DF(z)^\top$.
    \item Take a projected ascent step with step size $\eta_k>0$.
    \item Reinsert inactive points if desired, or simply keep only the active population.
\end{enumerate}

A self-contained support repository accompanying this report is available at

\begin{center}
\url{https://github.com/emmerichmtm/The-Magnitude-of-Dominated-Sets-ArXiV-Support}
\end{center}

\begin{remark}
For general problems one may add a small repulsion term or a proximal term in decision space to avoid collapse of multiple search points. In the present example the front is one-dimensional and the indicator is strictly concave in the ordered front parameters, so no extra regularization is required.
\end{remark}

\section{First biobjective test problem}

We first study a deliberately simple biobjective example for which the geometry can be analyzed almost completely. This makes it possible to compare the induced hypervolume-optimal and magnitude-optimal point distributions in closed form.

We consider the prescribed objectives
\[
F_1(x,y)=1-x^2,
\qquad
F_2(x,y)=1-(1-x^2)=x^2,
\qquad
(x,y)\in[-2,2]\times[-2,2].
\]
The second coordinate $y$ is neutral, so the Jacobian is
\[
DF(x,y)=
\begin{pmatrix}
-2x & 0\\
 2x & 0
\end{pmatrix}.
\]

\myparagraph{Efficient set and Pareto front}
\begin{proposition}
For the above problem, the efficient set is the whole search box
\[
E=[-2,2]\times[-2,2],
\]
and the Pareto front is the line segment
\[
\mathcal P=\{(1-t,t): 0\le t\le 4\}.
\]
\end{proposition}

\begin{proof}
The objectives depend only on $x$ through $t=x^2\in[0,4]$, and every search point has image $(1-t,t)$. If $t\neq s$, then $(1-t,t)$ and $(1-s,s)$ are incomparable in the product order because one has a larger first coordinate exactly when the other has a larger second coordinate. Thus no objective image dominates another distinct one, so every search point is efficient. The full set of objective images is exactly the stated segment.
\end{proof}

\begin{remark}
This example is degenerate in a useful way: the entire decision domain is efficient. Therefore the set-gradient method does not need to drive points \emph{onto} the Pareto set. Instead, it acts purely as a front-distribution mechanism, arranging a finite representative population along the Pareto front.
\end{remark}

\myparagraph{Anchor point and translated front}
To compute hypervolume and magnitude, we choose the natural anchor point
\[
r=(-3,0),
\]
which is the south-west corner of the objective image. Translating by $r$ gives the nonnegative coordinates
\[
\widetilde F(x,y)=F(x,y)-r=(4-x^2, x^2)=(4-t,t),
\qquad 0\le t\le 4.
\]
Hence the translated Pareto front is the segment from $(4,0)$ to $(0,4)$.

For a sorted sample
\[
0\le t_1\le\cdots\le t_\mu\le 4,
\]
the translated dominated region is
\[
\widetilde D=\bigcup_{i=1}^{\mu}[0,4-t_i]\times[0,t_i].
\]

\section{Indicator formulas on the front parameter}

To understand the two indicators on this first example, it is convenient to parametrize the front explicitly and derive the corresponding one-dimensional formulas. These formulas lead directly to the exact optimal distributions.

\myparagraph{Hypervolume}
For an ordered sample $t_1\le\cdots\le t_\mu$ define $t_0:=0$. Then
\[
\HV(\widetilde D)=\sum_{i=1}^{\mu}(t_i-t_{i-1})(4-t_i).
\]
Differentiating gives
\[
\frac{\partial \HV}{\partial t_i}=t_{i-1}+t_{i+1}-2t_i,
\qquad i=1,\dots,\mu-1,
\]
with the boundary derivative
\[
\frac{\partial \HV}{\partial t_\mu}=4-2t_\mu+t_{\mu-1},
\]
where $t_0:=0$.

\myparagraph{Magnitude}
Since the translated dominated region is zero-anchored,
\[
\Mag(\widetilde D)=1+\frac{X+Y}{2}+\frac{\HV(\widetilde D)}{4},
\]
where
\[
X=4-t_1,
\qquad
Y=t_\mu.
\]
Therefore
\[
J_{\mathrm{Mag}}(t_1,\dots,t_\mu)
=1+\frac{4-t_1+t_\mu}{2}+\frac{\HV(\widetilde D)}{4},
\]
and hence
\[
\frac{\partial J_{\mathrm{Mag}}}{\partial t_1}=-\frac12+\frac14\frac{\partial \HV}{\partial t_1},
\qquad
\frac{\partial J_{\mathrm{Mag}}}{\partial t_\mu}=\frac12+\frac14\frac{\partial \HV}{\partial t_\mu},
\]
while for $2\le i\le \mu-1$,
\[
\frac{\partial J_{\mathrm{Mag}}}{\partial t_i}=\frac14\frac{\partial \HV}{\partial t_i}.
\]

\myparagraph{Pull-back to the decision variable}
On the representative branch $x\in[0,2]$ we have $t=x^2$ and $dt/dx=2x$. Therefore
\[
\frac{dJ}{dx_i}=2x_i\frac{dJ}{dt_i},
\]
so a correct projected ascent step for the user-specified problem is
\[
x_{i}^{k+1}=\Pi_{[0,2]}\Bigl(x_i^k+\eta_k\,2x_i^k\,\frac{\partial J}{\partial t_i}(t_1^k,\dots,t_\mu^k)\Bigr),
\qquad
y_i^{k+1}=y_i^k.
\]
After each step the points are re-sorted by $t_i=(x_i^k)^2$. This is exactly the normalized set-gradient pull-back formula above, specialized to the Jacobian of the present problem.

\section{Exact optimal \texorpdfstring{$\mu$}{mu}-distributions}

We now solve the first example exactly at the level of indicator-optimal populations. This also places the present calculations next to the literature on exact optimal distributions for linear fronts.

For linear fronts, exact or analytically characterized hypervolume-optimal distributions have been studied extensively in two and three objectives; see, for example, Brockhoff~\cite{Brockhoff2010}, Ishibuchi et al.~\cite{IshibuchiImadaMasuyamaNojima2019}, Singh~\cite{Singh2025}, and Ishibuchi et al.~\cite{IshibuchiNanPang2025}. The formulas below recover the corresponding two-objective linear-front pattern in the present translated setting.
\begin{proposition}[hypervolume maximizer on the translated front]
For fixed population size $\mu$, hypervolume is uniquely maximized by the arithmetic progression
\[
t_i^{\mathrm{HV}}=\frac{4i}{\mu+1},
\qquad i=1,\dots,\mu.
\]
\end{proposition}

\begin{proof}
The first-order equations are
\[
2t_i=t_{i-1}+t_{i+1},
\quad i=1,\dots,\mu-1,
\qquad
2t_\mu=4+t_{\mu-1},
\]
with $t_0=0$. Hence the maximizer is an arithmetic progression $t_i=ic$, and the last equation gives $(\mu+1)c=4$.
\end{proof}

\begin{proposition}[magnitude maximizer on the translated front]
For fixed population size $\mu$, magnitude is uniquely maximized by including both extreme points and spacing the remaining points uniformly between them:
\[
t_i^{\mathrm{Mag}}=\frac{4(i-1)}{\mu-1},
\qquad i=1,\dots,\mu.
\]
\end{proposition}

\begin{proof}
The linear term $(4-t_1+t_\mu)/2$ is increased by decreasing $t_1$ and increasing $t_\mu$, so the maximizer must satisfy $t_1=0$ and $t_\mu=4$. With these endpoints fixed, the remaining variables maximize hypervolume on the shortened interval, hence form an arithmetic progression.
\end{proof}

For $\mu=8$ this yields
\[
\boxed{t_i^{\mathrm{HV}}=\frac{4i}{9},\quad i=1,\dots,8}
\qquad\text{and}\qquad
\boxed{t_i^{\mathrm{Mag}}=\frac{4(i-1)}{7},\quad i=1,\dots,8.}
\]
On the representative branch $x\in[0,2]$ these become
\[
x_i^{\mathrm{HV}}=2\sqrt{\frac{i}{9}},
\qquad
x_i^{\mathrm{Mag}}=2\sqrt{\frac{i-1}{7}}.
\]
Numerically,
\[
(x_i^{\mathrm{HV}})_{i=1}^8\approx(0.6667,0.9428,1.1547,1.3333,1.4907,1.6330,1.7638,1.8856),
\]
\[
(x_i^{\mathrm{Mag}})_{i=1}^8\approx(0.0000,0.7559,1.0690,1.3093,1.5119,1.6903,1.8516,2.0000).
\]

\section{Execution of the set-gradient ascent for \texorpdfstring{$\mu=8$}{mu=8}}

After the exact formulas, it is useful to verify that the projected set-gradient method reproduces the same distributions numerically. We therefore report one representative execution for a fixed population size.

We executed the projected set-gradient ascent on the representative branch $x\in[0,2]$ with $y_i\equiv0$, starting from the initial population
\[
(x_i^0)_{i=1}^8\approx(0.2076,0.3903,0.7841,1.0471,1.2343,1.6137,1.8890,1.9537).
\]
The step size schedule was
\[
\eta_k=0.08\cdot 0.9995^k,
\]
with re-sorting after every iteration and $5000$ iterations in total.

The resulting terminal values were
\[
(x_i^{\mathrm{HV,run}})_{i=1}^8\approx(0.6667,0.9428,1.1547,1.3333,1.4907,1.6330,1.7638,1.8856),
\]
\[
(x_i^{\mathrm{Mag,run}})_{i=1}^8\approx(0.0000,0.7559,1.0690,1.3093,1.5119,1.6903,1.8516,2.0000),
\]
which agree with the exact formulas above up to the displayed precision. In other words, on this problem the projected set-gradient ascent reproduces the exact indicator-maximizing populations.

\begin{figure}[p]
\centering
\begin{tikzpicture}
\begin{axis}[
width=12.8cm,height=6.8cm,
xlabel={$x$},ylabel={$y$},
xmin=-2.2,xmax=2.2,ymin=-2.2,ymax=2.2,
axis lines=middle,
axis on top,
xlabel style={at={(axis description cs:0.98,0.53)},anchor=west,fill=white,inner sep=1pt},
ylabel style={at={(axis description cs:0.52,0.98)},anchor=south,fill=white,inner sep=1pt},
legend style={draw=none,fill=none,at={(0.5,-0.20)},anchor=north,legend columns=2,/tikz/every even column/.append style={column sep=0.7cm}},
clip=false]
\path[fill=softgray,opacity=0.45] (axis cs:-2,-2) rectangle (axis cs:2,2);
\addlegendentry{efficient set $E=[-2,2]^2$}
\addplot[black,thick,dashed,domain=-2:2,samples=2] {0};
\addlegendentry{representative section $y=0$}
\addplot[only marks,mark=*,mark size=2.7pt,hvblue] coordinates {
(0.666667,0) (0.942809,0) (1.154701,0) (1.333333,0)
(1.490712,0) (1.632993,0) (1.763834,0) (1.885618,0)};
\addlegendentry{HV-optimal $8$ points}
\addplot[only marks,mark=square*,mark size=2.8pt,magred] coordinates {
(0.000000,0) (0.755929,0) (1.069045,0) (1.309307,0)
(1.511858,0) (1.690309,0) (1.851640,0) (2.000000,0)};
\addlegendentry{Magnitude-optimal $8$ points}
\end{axis}
\end{tikzpicture}
\caption{Decision-space picture. The efficient set is the entire search box, because the specified objectives depend only on $x$. The blue and red populations are shown on the representative section $y=0$, which is sufficient because $y$ is neutral.}
\label{fig:decision}
\end{figure}
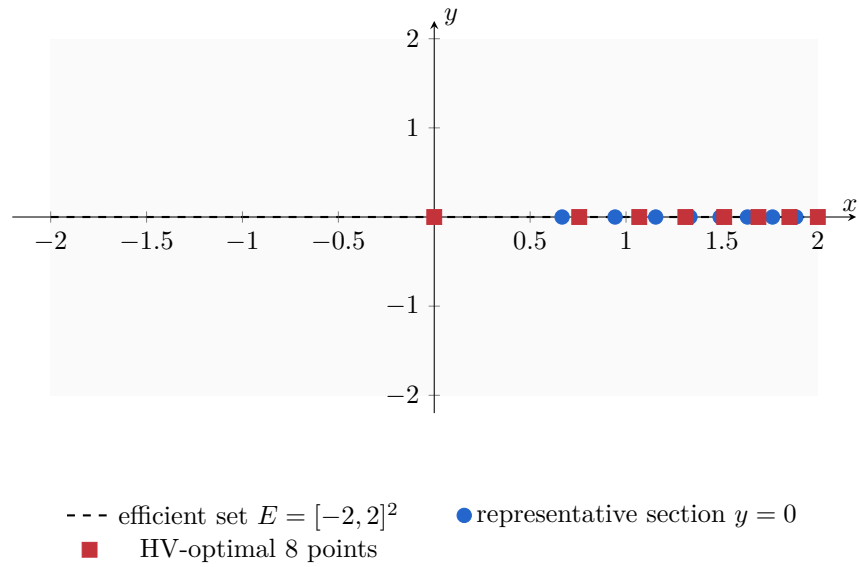

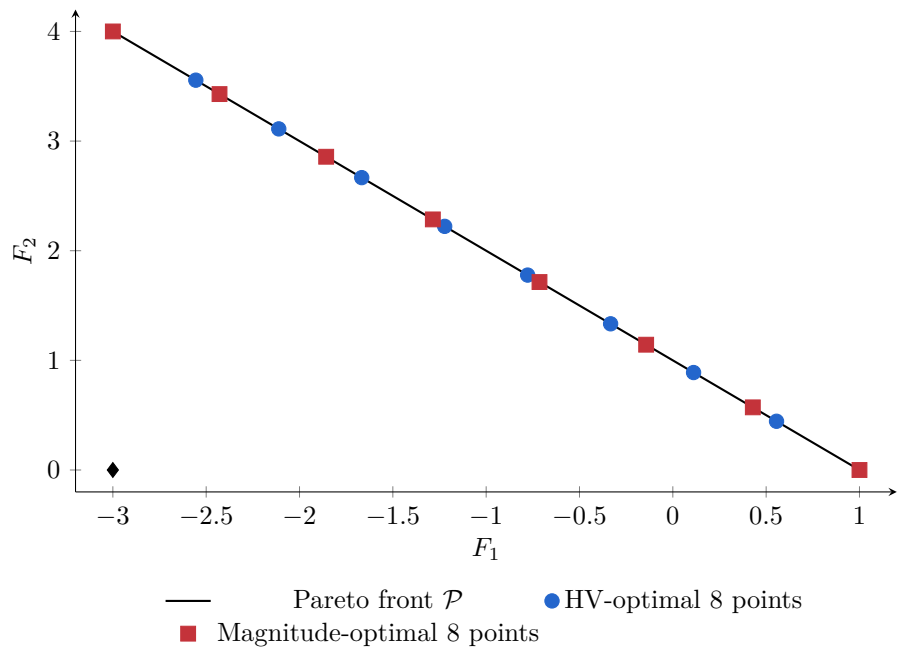
\begin{figure}[p]
\centering
\begin{tikzpicture}
\begin{axis}[
width=12.5cm,height=8cm,
xlabel={$F_1$},ylabel={$F_2$},
xmin=-3.2,xmax=1.2,ymin=-0.2,ymax=4.2,
axis lines=left,
legend style={draw=none,fill=none,at={(0.5,-0.18)},anchor=north,legend columns=2},
xlabel style={fill=white,inner sep=1pt},ylabel style={fill=white,inner sep=1pt}]
\addplot[black,thick,domain=0:4,samples=2] ({1-x},{x});
\addlegendentry{Pareto front $\mathcal P$}
\addplot[only marks,mark=*,mark size=2.7pt,hvblue] coordinates {
(0.555556,0.444444) (0.111111,0.888889) (-0.333333,1.333333) (-0.777778,1.777778)
(-1.222222,2.222222) (-1.666667,2.666667) (-2.111111,3.111111) (-2.555556,3.555556)};
\addlegendentry{HV-optimal $8$ points}
\addplot[only marks,mark=square*,mark size=2.8pt,magred] coordinates {
(1.000000,0.000000) (0.428571,0.571429) (-0.142857,1.142857) (-0.714286,1.714286)
(-1.285714,2.285714) (-1.857143,2.857143) (-2.428571,3.428571) (-3.000000,4.000000)};
\addlegendentry{Magnitude-optimal $8$ points}
\addplot[only marks,mark=diamond*,mark size=2.8pt,black] coordinates {(-3,0)};
\node[anchor=north west,fill=white,inner sep=1pt] at (rel axis cs:-1.02,0.98) {$r=(-3,0)$};
\end{axis}
\end{tikzpicture}
\caption{Objective-space picture for the user-specified objectives. Hypervolume is computed with respect to the anchor point $r=(-3,0)$. Hypervolume (blue) leaves a gap to each extreme of the front; magnitude (red) includes both extremes because of the extra shadow term.}
\label{fig:front}
\end{figure}

\section{Second biobjective test problem with a nontrivial Pareto set}

The first example isolates the distributional effect of the indicators on a degenerate efficient set. We now turn to a nondegenerate problem, where the method must first approach the efficient set and then distribute points along the front.

To complement the degenerate first example, we consider the symmetric quadratic problem
\[
F_1(x,y)=1-(x-1)^2-y^2,
\qquad
F_2(x,y)=1-(x+1)^2-y^2,
\qquad
(x,y)\in[-2,2]^2.
\]

\myparagraph{Efficient set and Pareto front}
\begin{proposition}
For this problem the efficient set is
\[
E_2=\{(x,0): -1\le x\le 1\},
\]
and the Pareto front is
\[
\mathcal P_2=\bigl\{(1-(x-1)^2,\;1-(x+1)^2): -1\le x\le 1\bigr\}.
\]
\end{proposition}

\begin{proof}
For fixed $x$, both objectives decrease when $|y|$ increases, so every efficient point must satisfy $y=0$. On the line $y=0$, points with $x<-1$ are dominated by $x=-1$, and points with $x>1$ are dominated by $x=1$. For $x\in[-1,1]$, increasing $x$ improves $F_1$ and worsens $F_2$, so distinct points are mutually non-dominating. Hence the efficient set is exactly the segment stated above, and its image is the Pareto front.
\end{proof}

We again use a common anchor point, now chosen as
\[
r=(-4,-4).
\]
With this choice, the translated front is zero-anchored and the same $2$-dimensional magnitude formula applies to the dominated region.

\myparagraph{Correct set-gradient pull-back}
Here
\[
DF(x,y)=
\begin{pmatrix}
-2(x-1) & -2y\\
-2(x+1) & -2y
\end{pmatrix}.
\]
For an active population point $z^{(i)}=(x_i,y_i)$ with translated objective vector $\widetilde y^{(i)}$, the correct ascent direction is
\[
\nabla_{z^{(i)}}J(P)=DF(x_i,y_i)^\top \nabla_{\widetilde y^{(i)}}J(P).
\]
Thus both coordinates are updated. In particular, the $y$-component is
\[
\frac{\partial J}{\partial y_i}
=
-2y_i\left(\frac{\partial J}{\partial \widetilde y^{(i)}_1}
+
\frac{\partial J}{\partial \widetilde y^{(i)}_2}\right),
\]
which vanishes only on the Pareto set $y=0$ (or when the objective-space derivative vanishes). This is the mechanism by which the method drives points toward the efficient set in this second problem.

\myparagraph{Numerical execution for $\mu=8$}
We executed the projected set-gradient ascent with $\mu=8$ active points, starting from a random population in the search box and projecting after each step onto $[-2,2]^2$. For stability we used backtracking on a small initial step size, normalized each active pull-back direction to unit Euclidean length, and re-sorted the active objective vectors after every update. The terminal populations on the Pareto set were as follows.

For hypervolume maximization:
\[
\begin{aligned}
x^{\mathrm{HV}} &\approx (-0.8195,\,-0.5526,\,-0.3221,\,-0.1060,\\
&\qquad\quad\ 0.1060,\,0.3221,\,0.5526,\,0.8195),\\
y^{\mathrm{HV}}&\approx 0.
\end{aligned}
\]

For magnitude maximization:
\[
\begin{aligned}
x^{\mathrm{Mag}} &\approx (-0.9153,\,-0.6069,\,-0.3517,\,-0.1155,\\
&\qquad\quad\ 0.1155,\,0.3517,\,0.6069,\,0.9153),\\
y^{\mathrm{Mag}}&\approx 0.
\end{aligned}
\]

Thus the second problem shows both aspects of the method: the pull-back drives the population onto the Pareto set, and then the two indicators distribute the points differently along the front. As in the first example, magnitude places more emphasis on the front extremes than hypervolume.

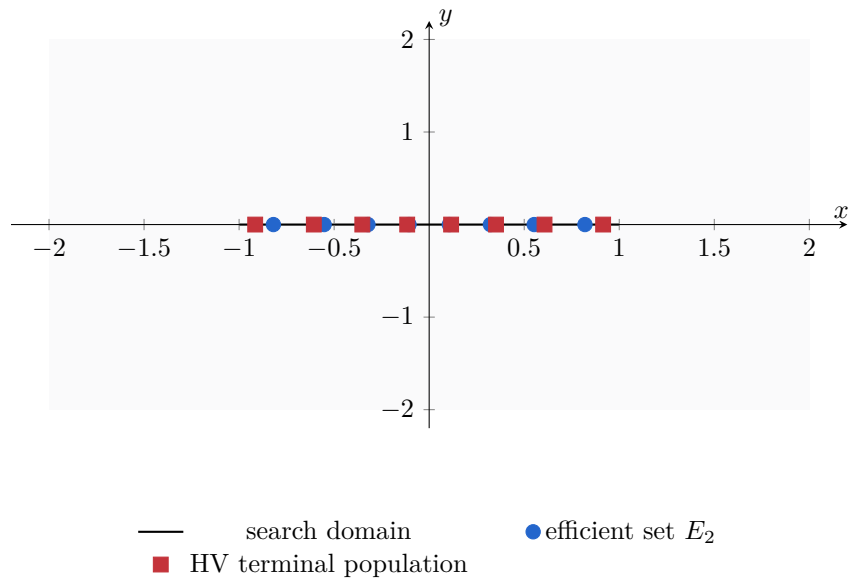
\begin{figure}[p]
\centering
\begin{tikzpicture}
\begin{axis}[
width=12.7cm,height=7.0cm,
xlabel={$x$},ylabel={$y$},
xmin=-2.2,xmax=2.2,ymin=-2.2,ymax=2.2,
axis lines=middle,
axis on top,
xlabel style={at={(axis description cs:0.98,0.53)},anchor=west,fill=white,inner sep=1pt},
ylabel style={at={(axis description cs:0.52,0.98)},anchor=south,fill=white,inner sep=1pt},
legend style={draw=none,fill=none,at={(0.5,-0.20)},anchor=north,legend columns=2,/tikz/every even column/.append style={column sep=0.7cm}},
clip=false]
\path[fill=softgray,opacity=0.45] (axis cs:-2,-2) rectangle (axis cs:2,2);
\addlegendentry{search domain}
\addplot[black,thick,domain=-1:1,samples=2] {0};
\addlegendentry{efficient set $E_2$}
\addplot[only marks,mark=*,mark size=2.7pt,hvblue] coordinates {
(-0.819514,0) (-0.552617,0) (-0.322145,0) (-0.106035,0)
(0.106035,0) (0.322145,0) (0.552617,0) (0.819514,0)};
\addlegendentry{HV terminal population}
\addplot[only marks,mark=square*,mark size=2.8pt,magred] coordinates {
(-0.915292,0) (-0.606902,0) (-0.351737,0) (-0.115523,0)
(0.115523,0) (0.351737,0) (0.606902,0) (0.915292,0)};
\addlegendentry{Magnitude terminal population}
\end{axis}
\end{tikzpicture}
\caption{Decision-space picture for the second problem. In contrast with the first example, the efficient set is the line segment $\{(x,0):-1\le x\le 1\}$, so the set-gradient ascent genuinely drives points onto a lower-dimensional Pareto set before redistributing them.}
\label{fig:decision2}
\end{figure}

\begin{figure}[p]
\centering
\begin{tikzpicture}
\begin{axis}[
width=12.4cm,height=8cm,
xlabel={$F_1$},ylabel={$F_2$},
xmin=-4.2,xmax=1.2,ymin=-4.2,ymax=1.2,
axis lines=left,
legend style={draw=none,fill=none,at={(0.5,-0.18)},anchor=north,legend columns=2},
xlabel style={fill=white,inner sep=1pt},ylabel style={fill=white,inner sep=1pt}]
\addplot[black,thick,domain=-1:1,samples=200] ({1-(x-1)^2},{1-(x+1)^2});
\addlegendentry{Pareto front $\mathcal P_2$}
\addplot[only marks,mark=*,mark size=2.7pt,hvblue] coordinates {
(-2.310632,0.967424) (-1.410625,0.799844) (-0.748052,0.540527) (-0.224106,0.200173)
(0.200173,-0.224106) (0.540527,-0.748052) (0.799844,-1.410625) (0.967424,-2.310632)};
\addlegendentry{HV terminal population}
\addplot[only marks,mark=square*,mark size=2.8pt,magred] coordinates {
(-2.668344,0.992825) (-1.582136,0.834947) (-0.826176,0.579718) (-0.244392,0.217698)
(0.217698,-0.244392) (0.579718,-0.826176) (0.834947,-1.582136) (0.992825,-2.668344)};
\addlegendentry{Magnitude terminal population}
\addplot[only marks,mark=diamond*,mark size=2.8pt,black] coordinates {(-4,-4)};
\node[anchor=north west,fill=white,inner sep=1pt] at (rel axis cs:0.02,0.98) {$r=(-4,-4)$};
\end{axis}
\end{tikzpicture}
\caption{Objective-space picture for the second problem. Again the magnitude-maximizing population reaches farther toward the extreme ends of the front than the hypervolume-maximizing one.}
\label{fig:front2}
\end{figure}
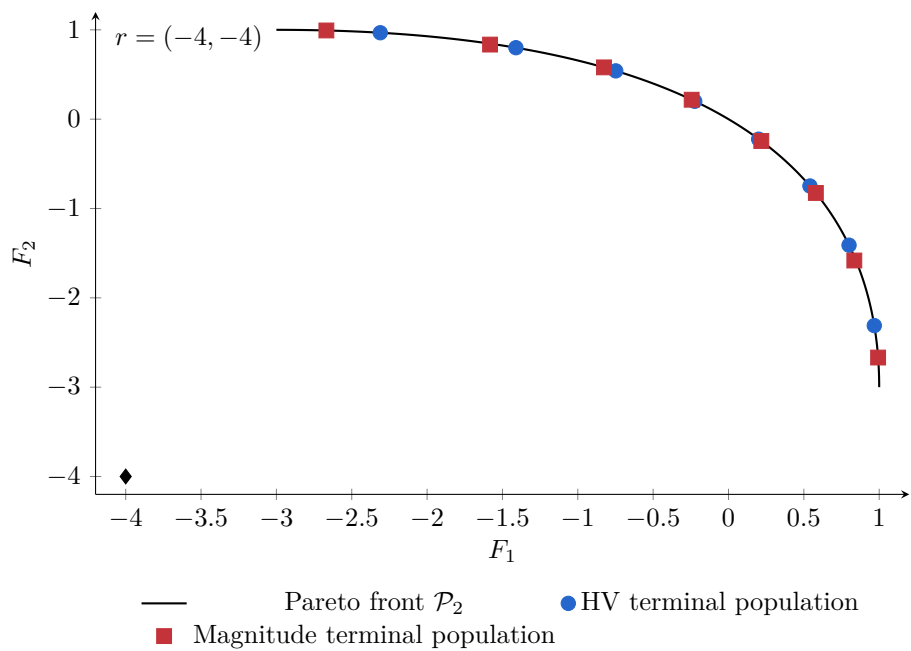

\section{Route to higher dimensions}

Before turning to explicit three-dimensional examples, we briefly indicate how the valuation picture extends beyond the planar case. The point of this section is not to give a complete theory, but to make the higher-dimensional pattern transparent.

For a single axis-parallel box
\[
B=\prod_{i=1}^n [0,L_i]\subset \ell_1^n,
\]
one always has
\[
\Mag(B)=\prod_{i=1}^n\Bigl(1+\frac{L_i}{2}\Bigr)=\sum_{k=0}^n \frac{e_k(L_1,\dots,L_n)}{2^k}.
\]
This already shows the general pattern: hypervolume is the top-dimensional term, while magnitude adds all lower-dimensional coordinate-shadow terms. The route from here to unions is the familiar one: pass from boxes to suitable unions via valuation arguments, and then to more general $\ell_1$-convex bodies by approximation. The symmetric simplex orbit above gives a first concrete three-dimensional illustration of how those extra shadow terms can change the preferred population geometry. For the purposes of indicator-based optimization, this suggests a family of ``magnitude-like'' indicators in higher dimensions that combine hypervolume with lower-dimensional shadow information. Analytical hypervolume calculations on structured fronts, algorithmic work on efficient hypervolume computation, and higher-order derivative information suggest one possible route toward efficient magnitude computations in higher dimension~\cite{Singh2025,GuerreiroFonsecaPaquete2021,DeutzEmmerichWang2023}.

\section{Three-dimensional simplex example}
\label{sec:3dexample}

We now pass from the planar setting to a genuinely three-dimensional benchmark. The simplex geometry is sufficiently symmetric to permit explicit calculations, while still revealing a clear contrast between hypervolume and magnitude.

We now add a first genuinely three-dimensional indicator comparison. Consider the unit simplex
\[
\Delta_2=\{(z_1,z_2,z_3)\in \R_{\ge 0}^3: z_1+z_2+z_3=1\}
\]
with the anchor point fixed at the origin. For a six-point population, a natural fully symmetric ansatz is the orbit of one ordered triple
\[
u\ge v\ge w\ge 0,\qquad u+v+w=1,
\]
under all six coordinate permutations:
\[
P(u,v,w)=\mathfrak S_3\cdot (u,v,w).
\]
The associated dominated region is
\[
D(u,v,w)=\bigcup_{p\in P(u,v,w)} [0,p]\subset \R_{\ge 0}^3.
\]
This gives a tidy first benchmark because the geometry is symmetric and all formulas can be written in closed form. It is also natural from the viewpoint of optimal-distribution studies for hypervolume and related indicators on low-dimensional Pareto fronts~\cite{IshibuchiImadaMasuyamaNojima2019,IshibuchiNanPang2025}.

\myparagraph{Hypervolume and magnitude formulas on the symmetric orbit}

\begin{proposition}[3D hypervolume on the six-point symmetric orbit]
For the dominated region $D(u,v,w)$ generated by $P(u,v,w)$ one has
\[
\HV(D(u,v,w))=6uvw-3uw^2-3v^2w+w^3.
\]
\end{proposition}

\begin{proof}
By symmetry, every point in the union has sorted coordinates bounded by $(u,v,w)$ after a suitable coordinate permutation. Equivalently, the union can be partitioned into six congruent chambers according to the coordinate ordering, and inclusion-exclusion over the nested threshold boxes yields
\[
\vol_3(D)=6uvw-3uw^2-3v^2w+w^3.
\]
A direct check confirms that this formula is exact for every ordered triple $u\ge v\ge w\ge 0$ with $u+v+w=1$.
\end{proof}

\begin{proposition}[3D magnitude on the same orbit]
For the same dominated region one has
\[
\Mag(D(u,v,w))
=
1+\frac{3u}{2}
+\frac{3(2uv-v^2)}{4}
+\frac{6uvw-3uw^2-3v^2w+w^3}{8}.
\]
\end{proposition}

\begin{proof}
In the $\ell_1$ box setting, magnitude is the sum of positive shadow terms. Here each one-dimensional coordinate projection is the interval $[0,u]$, so the total one-dimensional shadow contribution is $V_1=3u$. Each coordinate-plane projection is the union of the two rectangles $[0,u]\times[0,v]$ and $[0,v]\times[0,u]$, hence has area $2uv-v^2$; summing over the three coordinate planes gives $V_2=3(2uv-v^2)$. The three-dimensional term is precisely the hypervolume above. Substituting these into
\[
\Mag(D)=1+\frac{V_1}{2}+\frac{V_2}{4}+\frac{V_3}{8}
\]
gives the claimed formula.
\end{proof}

\myparagraph{Two optimal six-point symmetric configurations}

The next result should be read exactly as stated: it optimizes \emph{within the six-point symmetric orbit family} $P(u,v,w)$.

\begin{theorem}[Hypervolume-optimal and magnitude-optimal symmetric six-point orbits]
Within the symmetric family $P(u,v,w)$ on $\Delta_2$ with $u\ge v\ge w\ge 0$ and $u+v+w=1$, the hypervolume indicator is maximized at
\[
u_{\mathrm{HV}}=\frac{62+5\sqrt{13}}{153}\approx 0.523057,\qquad
v_{\mathrm{HV}}=\frac{43+\sqrt{13}}{153}\approx 0.304611,
\]
\[
w_{\mathrm{HV}}=\frac{48-6\sqrt{13}}{153}\approx 0.172331.
\]
Hence the hypervolume-maximizing symmetric population is the six-point orbit of
\[
(u_{\mathrm{HV}},v_{\mathrm{HV}},w_{\mathrm{HV}}).
\]

Among distinct symmetric six-point orbits, the magnitude indicator is maximized at the boundary point
\[
(u_{\mathrm{Mag}},v_{\mathrm{Mag}},w_{\mathrm{Mag}})
=
\Bigl(\frac79,\frac29,0\Bigr).
\]
Hence the magnitude-maximizing symmetric population is the six-point orbit of
\[
\Bigl(\frac79,\frac29,0\Bigr).
\]
\end{theorem}

\begin{proof}
For hypervolume, substitute $w=1-u-v$ into the cubic formula and solve the stationarity equations in the feasible region $u\ge v\ge w\ge 0$. This yields the interior critical point above, which is the unique feasible maximizer.

For magnitude, substitute $w=1-u-v$ into the magnitude formula. The interior stationarity equations do not produce the best feasible value among distinct six-point symmetric orbits; the maximum is attained on the boundary $w=0$, where the problem reduces to a one-variable quadratic in $v$ with $u=1-v$. Maximizing that boundary expression gives $v=2/9$ and $u=7/9$.
\end{proof}

\begin{remark}
The contrast is instructive. Hypervolume prefers a genuinely three-dimensional interior orbit with all coordinates positive. Magnitude, by contrast, values the lower-dimensional shadows strongly enough that the optimum within this symmetric family moves to the simplex boundary. This is the three-dimensional analogue of the stronger emphasis on extreme points already seen in the planar examples.
\end{remark}

\myparagraph{Explicit coordinate lists}

For hypervolume, the six points are the permutations of
\[
(0.523057,\ 0.304611,\ 0.172331).
\]
For magnitude, the six points are the permutations of
\[
\Bigl(\frac79,\ \frac29,\ 0\Bigr).
\]
That is,
\[
\begin{aligned}
P_{\mathrm{Mag}}=\{&
(7/9,2/9,0),\ (7/9,0,2/9),\ (2/9,7/9,0),\\
&(2/9,0,7/9),\ (0,7/9,2/9),\ (0,2/9,7/9)\}.
\end{aligned}
\]

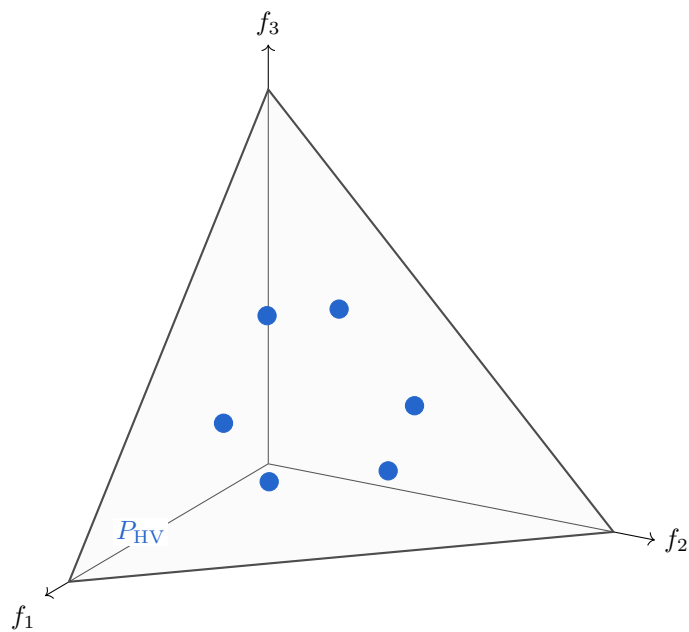
\begin{figure}[p]
\centering
\tdplotsetmaincoords{70}{120}
\begin{tikzpicture}[tdplot_main_coords,scale=5.3]
\coordinate (O) at (0,0,0);
\coordinate (E1) at (1,0,0);
\coordinate (E2) at (0,1,0);
\coordinate (E3) at (0,0,1);

\draw[->,thin] (O) -- (1.12,0,0) node[anchor=north east] {$f_1$};
\draw[->,thin] (O) -- (0,1.12,0) node[anchor=west] {$f_2$};
\draw[->,thin] (O) -- (0,0,1.12) node[anchor=south] {$f_3$};

\filldraw[fill=softgray,draw=black!65,opacity=0.35] (E1) -- (E2) -- (E3) -- cycle;
\draw[black!70,thick] (E1) -- (E2) -- (E3) -- cycle;

\foreach \P in {(0.523057,0.304611,0.172331),
                (0.523057,0.172331,0.304611),
                (0.304611,0.523057,0.172331),
                (0.304611,0.172331,0.523057),
                (0.172331,0.523057,0.304611),
                (0.172331,0.304611,0.523057)}{
  \filldraw[hvblue] \P circle (0.65pt);
}
\node[hvblue,anchor=north,fill=white,inner sep=1pt] at (0.50,-0.08,0) {$P_{\mathrm{HV}}$};
\end{tikzpicture}
\caption{A 3D simplex plot of the hypervolume-optimal symmetric six-point orbit on $\Delta_2$. The six points lie strictly in the interior of the simplex edges and faces, reflecting the genuinely three-dimensional nature of hypervolume maximization.}
\label{fig:3dhv}
\end{figure}

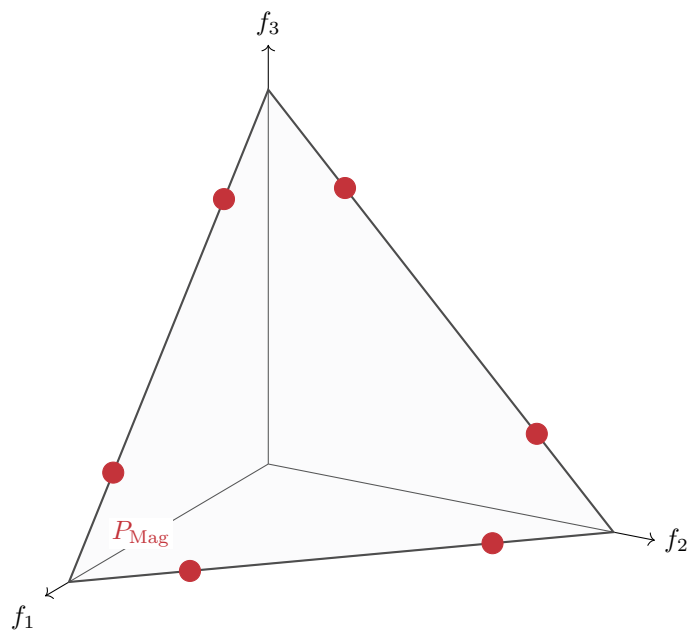
\begin{figure}[p]
\centering
\tdplotsetmaincoords{70}{120}
\begin{tikzpicture}[tdplot_main_coords,scale=5.3]
\coordinate (O) at (0,0,0);
\coordinate (E1) at (1,0,0);
\coordinate (E2) at (0,1,0);
\coordinate (E3) at (0,0,1);

\draw[->,thin] (O) -- (1.12,0,0) node[anchor=north east] {$f_1$};
\draw[->,thin] (O) -- (0,1.12,0) node[anchor=west] {$f_2$};
\draw[->,thin] (O) -- (0,0,1.12) node[anchor=south] {$f_3$};

\filldraw[fill=softgray,draw=black!65,opacity=0.35] (E1) -- (E2) -- (E3) -- cycle;
\draw[black!70,thick] (E1) -- (E2) -- (E3) -- cycle;

\foreach \P in {(0.777778,0.222222,0.000000),
                (0.777778,0.000000,0.222222),
                (0.222222,0.777778,0.000000),
                (0.222222,0.000000,0.777778),
                (0.000000,0.777778,0.222222),
                (0.000000,0.222222,0.777778)}{
  \filldraw[magred] \P circle (0.75pt);
}
\node[magred,anchor=north,fill=white,inner sep=1pt] at (0.50,-0.08,0) {$P_{\mathrm{Mag}}$};
\end{tikzpicture}
\caption{A 3D simplex plot of the magnitude-optimal symmetric six-point orbit on $\Delta_2$. In contrast with hypervolume, the magnitude optimum lies on the simplex boundary, emphasizing lower-dimensional shadow terms.}
\label{fig:3dmag}
\end{figure}

\section{Exact 3D inclusion--exclusion oracles and simplex results}

We next make the three-dimensional computations fully explicit. This section provides exact inclusion--exclusion formulas and exact numerical oracles for the simplex experiments.

\myparagraph{Exact inclusion--exclusion oracles and tie-shared subgradients}

For anchored boxes
\[
D(P)=\bigcup_{p\in P}[0,p]\subseteq [0,1]^3
\]
generated by a finite set $P=\{p^{(1)},\dots,p^{(\mu)}\}\subseteq \Delta_2$, both hypervolume and magnitude admit exact inclusion--exclusion formulas. For every nonempty index set $J\subseteq\{1,\dots,\mu\}$ define the componentwise minimum
\[
m(J)=\min_{j\in J} p^{(j)}=(m_1(J),m_2(J),m_3(J)).
\]
Then
\[
\HV(D(P))
=
\sum_{\emptyset\ne J\subseteq\{1,\dots,\mu\}} (-1)^{|J|+1}\,m_1(J)m_2(J)m_3(J),
\]
and, since the magnitude of one anchored box is
\[
\Mag([0,a]\times[0,b]\times[0,c])=\Bigl(1+\frac a2\Bigr)\Bigl(1+\frac b2\Bigr)\Bigl(1+\frac c2\Bigr),
\]
valuation gives
\[
\Mag(D(P))=
\sum_{\emptyset\ne J\subseteq\{1,\dots,\mu\}} (-1)^{|J|+1}
\prod_{r=1}^3\Bigl(1+\frac{m_r(J)}{2}\Bigr).
\]
For the small three-dimensional simplex populations considered below ($\mu=9$ and $\mu=10$), these formulas are fast enough to serve both as the optimization oracle and as the final verification oracle.

It is also convenient to separate the lower-dimensional contributions:
\[
\Mag(D(P))=
1+\frac{V_1(D(P))}{2}+\frac{V_2(D(P))}{4}+\frac{V_3(D(P))}{8},
\]
with
\[
V_1(D(P))=\max_i p^{(i)}_1+\max_i p^{(i)}_2+\max_i p^{(i)}_3,
\]
\begin{align*}
V_2(D(P))&=\operatorname{area}(\proj_{12} D(P))
+\operatorname{area}(\proj_{13} D(P))
+\operatorname{area}(\proj_{23} D(P)),\\
V_3(D(P))&=\HV(D(P)).
\end{align*}
Each two-dimensional projection area is again an anchored inclusion--exclusion sum.

For the gradient-based runs we use \emph{tie-shared exact subgradients}. For a subset $J$ and coordinate $\alpha\in\{1,2,3\}$ let
\[
M_\alpha(J)=\arg\min_{j\in J} p^{(j)}_\alpha .
\]
Then a symmetric exact subgradient of hypervolume is obtained by distributing the contribution of each intersection term equally over all tied minimizers:
\[
\boxed{\;
\partial_{p^{(i)}_\alpha}\HV(D(P))
=
\sum_{\substack{\emptyset\ne J\subseteq\{1,\dots,\mu\}\\ i\in M_\alpha(J)}}
(-1)^{|J|+1}
\frac{1}{|M_\alpha(J)|}
\prod_{\beta\ne \alpha} m_\beta(J). \;}
\]
Away from ties this reduces to the ordinary derivative. At ties it selects the most symmetric subgradient compatible with the permutation symmetry of the point set.

The same construction applies to magnitude. For the one-dimensional shadow term we share the derivative equally among all maximizers:
\[
\partial_{p^{(i)}_\alpha}V_1(D(P))
=
\begin{cases}
1/|\arg\max_j p^{(j)}_\alpha|, & \text{if }p^{(i)}_\alpha=\max_j p^{(j)}_\alpha,\\[0.5ex]
0, & \text{otherwise.}
\end{cases}
\]
The projection areas are handled by the two-dimensional analogue of the hypervolume subgradient. Consequently,
\[
\boxed{\;
\nabla_{p^{(i)}}\Mag(D(P))
=
\frac12\,\nabla_{p^{(i)}}V_1
+\frac14\bigl(\nabla_{p^{(i)}}A_{12}+\nabla_{p^{(i)}}A_{13}+\nabla_{p^{(i)}}A_{23}\bigr)
+\frac18\,\nabla_{p^{(i)}}\HV, \;}
\]
where $A_{\alpha\beta}$ denotes the exact anchored union area in the corresponding coordinate plane.

In the three-dimensional simplex experiments we update every point by
\[
p^{(i)}_{t+1}
=
\Pi_{\Delta_2}\!\left(p^{(i)}_t+\eta_t\,\widehat g^{(i)}_t\right),
\qquad
\widehat g^{(i)}_t
=
\frac{\Pi_T g^{(i)}_t}{\|\Pi_T g^{(i)}_t\|_2},
\]
whenever $\Pi_T g^{(i)}_t\neq 0$, where $g^{(i)}_t$ is the exact tie-shared subgradient of the chosen indicator and
\[
\Pi_T g = g-\frac{g_1+g_2+g_3}{3}(1,1,1)
\]
is the orthogonal projection onto the simplex tangent space. The step size $\eta_t$ is chosen by backtracking line search and $\Pi_{\Delta_2}$ denotes Euclidean projection onto the simplex. This is the clean projected normalized-gradient scheme used for all recomputed three-dimensional numerical results in the present revision.
\myparagraph{Complete 3D Das--Dennis grids and a closed-form magnitude formula}

For an integer $H\ge 1$, define the complete Das--Dennis grid on the simplex, following the standard simplex-lattice construction associated with Das and Dennis~\cite{DasDennis1998},
\[
G_H=\left\{\left(\frac{i}{H},\frac{j}{H},\frac{k}{H}\right): i,j,k\in\mathbb N_0,\ i+j+k=H\right\}.
\]
Its cardinality is $|G_H|=\binom{H+2}{2}$. Let
\[
U_H=\bigcup_{p\in G_H}[0,p].
\]

\begin{proposition}[staircase structure of $U_H$]
Partition $[0,1]^3$ into the grid cells
\[
C_{abc}=\Bigl[\frac{a-1}{H},\frac{a}{H}\Bigr]\times \Bigl[\frac{b-1}{H},\frac{b}{H}\Bigr]\times \Bigl[\frac{c-1}{H},\frac{c}{H}\Bigr],
\qquad 1\le a,b,c\le H.
\]
Then $C_{abc}\subseteq U_H$ if and only if $a+b+c\le H$.
\end{proposition}

\begin{proof}
The cell $C_{abc}$ is covered if and only if its upper corner $(a/H,b/H,c/H)$ is dominated by some point $(i/H,j/H,k/H)\in G_H$. This requires $i\ge a$, $j\ge b$, $k\ge c$, and $i+j+k=H$. Such nonnegative integers exist exactly when $a+b+c\le H$.
\end{proof}

\begin{corollary}[exact volume and shadow areas]
For $H\ge 1$,
\[
\HV(U_H)=\frac{\binom{H}{3}}{H^3}=\frac{(H-1)(H-2)}{6H^2},
\]
with the convention $\binom{H}{3}=0$ for $H<3$. Moreover each two-dimensional coordinate projection of $U_H$ is the staircase region under the line $x+y\le 1$ on the same mesh, so
\[
\operatorname{area}(\proj_{12}U_H)=\operatorname{area}(\proj_{13}U_H)=\operatorname{area}(\proj_{23}U_H)=\frac{\binom{H}{2}}{H^2}=\frac{H-1}{2H}.
\]
Finally, each one-dimensional coordinate projection is the full interval $[0,1]$.
\end{corollary}

\begin{proof}
The three-dimensional cells counted in the proposition above are in bijection with positive integer triples $(a,b,c)$ with $a+b+c\le H$, whose number is $\binom{H}{3}$. Each cell has volume $H^{-3}$. The projection statement is the two-dimensional analogue: a projected cell with upper index pair $(a,b)$ is covered if and only if $a+b\le H$, giving $\binom{H}{2}$ covered cells, each of area $H^{-2}$. The interval projections are obvious from the vertices of the simplex grid.
\end{proof}

\begin{theorem}[closed form for complete 3D Das--Dennis grids]
For every integer $H\ge 1$,
\[
\boxed{\;
\Mag(U_H)=\frac52+\frac{3(H-1)}{8H}+\frac{(H-1)(H-2)}{48H^2}.\;}
\]
Equivalently,
\[
\Mag(U_H)=1+\frac{3}{2}+\frac{3}{4}\cdot \frac{H-1}{2H}+\frac{1}{8}\cdot \frac{\binom{H}{3}}{H^3}.
\]
In particular,
\[
\Mag(U_2)=2.6875,
\qquad
\Mag(U_3)=2.7546296296\ldots,
\qquad
\Mag(U_4)=2.7890625.
\]
\end{theorem}

\begin{proof}
The set $U_H$ is contractible and lies in the anchored $\ell_1$ class discussed above, so
\[
\Mag(U_H)=1+\frac{V_1(U_H)}{2}+\frac{V_2(U_H)}{4}+\frac{V_3(U_H)}{8}.
\]
By the corollary, $V_1(U_H)=3$, $V_2(U_H)=3(H-1)/(2H)$, and $V_3(U_H)=\binom{H}{3}/H^3$. Substituting these values gives the stated formula.
\end{proof}

\begin{remark}
The theorem explains the exact values observed in our runs for the level-$2$ and level-$3$ complete grids. It also shows that complete Das--Dennis grids interpolate between a boundary-heavy regime and a limiting value approaching $2.885416\ldots$ as $H\to\infty$.
\end{remark}

\myparagraph{A projected-stationarity phenomenon for magnitude}

The exact $9$-point and $10$-point experiments suggested something stronger than mere lack of movement: complete Das--Dennis grids appear to be \emph{projected-stationary} for magnitude on the simplex.

\begin{definition}
Let $P\subset\Delta_2$ be a finite point set. We call $P$ \emph{projected-stationary for magnitude} if, for every non-vertex point $p\in P$, every feasible tangent perturbation $d$ with $d_1+d_2+d_3=0$ and $p+\varepsilon d\in\Delta_2$ for small $\varepsilon$, the first variation of
\[
\varepsilon\longmapsto \Mag\bigl(D(P\setminus\{p\}\cup\{p+\varepsilon d\})\bigr)
\]
vanishes at $\varepsilon=0$, while at each simplex vertex every feasible inward perturbation has nonpositive one-sided first variation.
\end{definition}

\myparagraph{Exact numerical evidence.}
Using exact inclusion--exclusion, we checked the following perturbations.

\begin{itemize}[leftmargin=2em]
\item For the centroid $p=(1/3,1/3,1/3)\in G_3$ and tangent direction $d=(1,-1,0)$,
\[
\Mag\bigl(D((G_3\setminus\{p\})\cup\{p+\varepsilon d\})\bigr)-\Mag(U_3)
\approx -\frac{1}{24}\,\varepsilon^2.
\]
For $\varepsilon=10^{-3},10^{-2},5\cdot 10^{-2}$ the observed coefficients $(\Delta \Mag)/\varepsilon^2$ were $-0.0416666674$, $-0.0416666667$, and $-0.0416666667$.

\item For the edge point $p=(1/3,2/3,0)\in G_3$ and the same tangent direction $d=(1,-1,0)$,
\[
\Mag\bigl(D((G_3\setminus\{p\})\cup\{p+\varepsilon d\})\bigr)-\Mag(U_3)
\approx -\frac{1}{4}\,\varepsilon^2.
\]
For the same three step sizes, the observed coefficients $(\Delta \Mag)/\varepsilon^2$ were approximately $-0.25$, $-0.25$, and $-0.25$.

\item For the edge midpoint $p=(1/2,1/2,0)\in G_2$ and tangent direction $d=(1,-1,0)$, the same quadratic law appears:
\[
\Mag\bigl(D((G_2\setminus\{p\})\cup\{p+\varepsilon d\})\bigr)-\Mag(U_2)
\approx -\frac{1}{4}\,\varepsilon^2.
\]

\item For the vertex $v=(1,0,0)\in G_3$ and inward direction $d=(-1,1/2,1/2)$, the change is already negative to first order. Numerically,
\[
\Mag\bigl(D((G_3\setminus\{v\})\cup\{v+\varepsilon d\})\bigr)-\Mag(U_3)
\]
was approximately $-4.17\times 10^{-4}$, $-4.19\times 10^{-3}$, and $-2.14\times 10^{-2}$ for $\varepsilon=10^{-3},10^{-2},5\cdot 10^{-2}$.
\end{itemize}

These checks strongly indicate that the complete grids $G_H$ are stationary in the correct constrained sense: interior and edge points have zero first variation in feasible tangent directions, while vertices satisfy the expected one-sided boundary optimality condition.

\begin{remark}[Conjectural explanation]
A natural proof strategy is to combine valuation with local cancellation. When one non-vertex grid point is perturbed tangentially, only inclusion--exclusion terms involving that point can change. Because the Das--Dennis grid is complete, neighboring anchored boxes exist on both sides of each tangent direction, and the first-order gains and losses in the affected cell pattern cancel. At a simplex vertex no compensating neighbor exists beyond the boundary, so the one-sided linear term is negative. We therefore formulate the following conjecture.
\end{remark}

\begin{conjecture}
For every $H\ge 2$, the complete Das--Dennis grid $G_H$ is projected-stationary for anchored-box magnitude on $\Delta_2$. More precisely:
\begin{enumerate}[leftmargin=2em]
\item if $p\in G_H$ is not a simplex vertex and $d$ is any feasible tangent direction through $p$, then
\[
\Mag\bigl(D((G_H\setminus\{p\})\cup\{p+\varepsilon d\})\bigr)
=
\Mag(U_H)-c_{p,d}\,\varepsilon^2+o(\varepsilon^2)
\]
for some $c_{p,d}\ge 0$;
\item if $v\in G_H$ is a simplex vertex and $d$ is any feasible inward direction, then
\[
\Mag\bigl(D((G_H\setminus\{v\})\cup\{v+\varepsilon d\})\bigr)
=\Mag(U_H)-\gamma_{v,d}\,\varepsilon+o(\varepsilon)
\]
with $\gamma_{v,d}>0$.
\end{enumerate}
\end{conjecture}

This conjecture would explain why the exact magnitude-ascent runs initialized at $G_3$ and at $G_3\setminus\{(1/3,1/3,1/3)\}$ remained at those regular patterns: they are not merely convenient seeds but appear to be intrinsic stationary reference distributions for the magnitude indicator.

\myparagraph{Exact $9$-point and $10$-point runs in 3D}

We return to the simplex front
\[
\Delta_2=\{z\in\R_{\ge 0}^3:\ z_1+z_2+z_3=1\},
\]
with anchor point $r=(0,0,0)$. For $\mu=10$ we initialize the projected ascent at the level-$3$ Das--Dennis grid
\[
\mathcal G_{10}=G_3,
\]
and for $\mu=9$ we remove the centroid,
\[
\mathcal G_9=G_3\setminus\{(1/3,1/3,1/3)\}.
\]
From these starts we run the exact projected normalized-gradient scheme from the preceding subsection, using the tie-shared inclusion--exclusion subgradients of hypervolume or magnitude and a backtracking line search. No finite-difference approximation is used in these three-dimensional runs.

The exact objective values are summarized in Table~\ref{tab:exact3d910}. Hypervolume now moves to fully symmetric terminal populations, while magnitude stays at the Das--Dennis initial patterns. The clean subgradient implementation therefore confirms the qualitative picture and provides the definitive three-dimensional numerical values reported in this paper.

\begin{table}[h]
\centering
\caption{Exact objective values for the $9$-point and $10$-point 3D simplex examples under the projected normalized-gradient method with tie-shared inclusion--exclusion subgradients.}
\label{tab:exact3d910}
\small
\renewcommand{\arraystretch}{1.18}
\begin{tabular}{p{1.7cm}p{3.6cm}p{3.2cm}p{4.2cm}}
\hline
Case & Initialization & Final value & Terminal pattern \\
\hline
$9$-pt HV & $\HV(\mathcal G_9)=0$ & $0.0752901021$ & vertices $+$ 6-point HV orbit \\
$10$-pt HV & $\HV(\mathcal G_{10})=1/27\approx 0.037037$ & $59/729\approx 0.0809327846$ & vertices $+$ centroid $+$ 6-point orbit \\
$9$-pt Mag & $\Mag(\mathcal G_9)=2.75$ & $2.75$ & unchanged Das--Dennis grid \\
$10$-pt Mag & $\Mag(\mathcal G_{10})=2.7546296296$ & $2.7546296296$ & unchanged Das--Dennis grid \\
\hline
\end{tabular}
\normalsize
\end{table}

The terminal hypervolume populations are as follows.

\myparagraph{$9$-point HV population}
The run converges to the three simplex vertices together with the six-point hypervolume-optimal symmetric orbit from Section~\ref{sec:3dexample},
\[
\{(1,0,0),(0,1,0),(0,0,1)\}\cup \mathfrak S_3\cdot(u^*,v^*,w^*),
\]
where
\begin{align*}
u^*&=\frac{62+5\sqrt{13}}{153}\approx 0.523057,\\
v^*&=\frac{43+\sqrt{13}}{153}\approx 0.304611,\\
w^*&=\frac{48-6\sqrt{13}}{153}\approx 0.172331.
\end{align*}
Since the vertex boxes have zero $3$D volume, the resulting hypervolume value is the same as for the six-point orbit:
\begin{align*}
\HV_{9\text{-pt}}&=6u^*v^*w^*-3u^*(w^*)^2-3(v^*)^2w^*+(w^*)^3\\
&\approx 0.0752901021.
\end{align*}

\myparagraph{$10$-point HV population}
The clean projected ascent converges to the centroid together with the three vertices and the six permutations of
\[
\left(\frac{15}{27},\frac{8}{27},\frac{4}{27}\right)
=
(0.555556,0.296296,0.148148).
\]
Thus
\[
P^{\mathrm{HV}}_{10}
=
\{(1,0,0),(0,1,0),(0,0,1),(1/3,1/3,1/3)\}
\cup
\mathfrak S_3\cdot\left(\frac{15}{27},\frac{8}{27},\frac{4}{27}\right),
\]
and exact inclusion--exclusion gives
\[
\HV(P^{\mathrm{HV}}_{10})=\frac{59}{729}\approx 0.0809327846.
\]

\myparagraph{$9$-point and $10$-point magnitude populations}
For both $\mu=9$ and $\mu=10$, no improving step is accepted by the backtracking line search from the Das--Dennis initialization. The terminal populations therefore remain exactly the level-$3$ Das--Dennis patterns introduced above:
\[
\mathcal G_9=G_3\setminus\{(1/3,1/3,1/3)\},
\qquad
\mathcal G_{10}=G_3.
\]

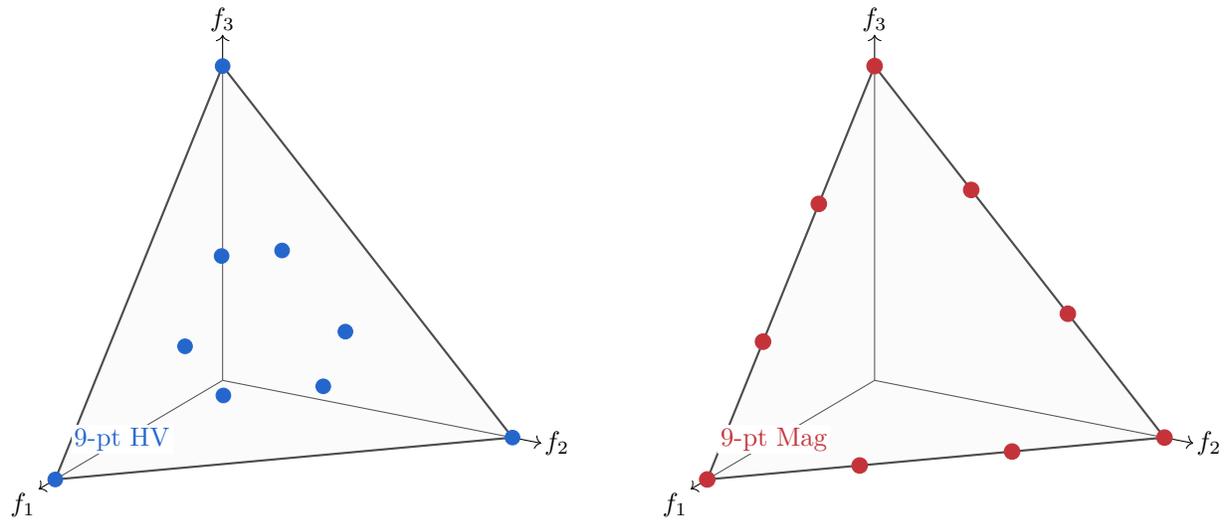
\begin{figure}[p]
\centering
\begin{minipage}{0.475\textwidth}
\centering
\tdplotsetmaincoords{70}{120}
\begin{tikzpicture}[tdplot_main_coords,scale=4.45]
\coordinate (O) at (0,0,0);
\draw[->,thin] (O) -- (1.10,0,0) node[anchor=north east,fill=white,inner sep=1pt] {$f_1$};
\draw[->,thin] (O) -- (0,1.10,0) node[anchor=west,fill=white,inner sep=1pt] {$f_2$};
\draw[->,thin] (O) -- (0,0,1.10) node[anchor=south,fill=white,inner sep=1pt] {$f_3$};
\coordinate (E1) at (1,0,0);
\coordinate (E2) at (0,1,0);
\coordinate (E3) at (0,0,1);
\filldraw[fill=softgray,draw=black!65,opacity=0.35] (E1) -- (E2) -- (E3) -- cycle;
\draw[black!70,thick] (E1) -- (E2) -- (E3) -- cycle;
\foreach \P in {(1,0,0),(0,1,0),(0,0,1),
                (0.523057,0.304611,0.172331),(0.523057,0.172331,0.304611),
                (0.304611,0.523057,0.172331),(0.304611,0.172331,0.523057),
                (0.172331,0.523057,0.304611),(0.172331,0.304611,0.523057)}{
  \filldraw[hvblue] \P circle (0.62pt);
}
\node[hvblue,anchor=north,fill=white,inner sep=1pt] at (0.48,-0.07,0) {$9$-pt HV};
\end{tikzpicture}
\end{minipage}\hfill
\begin{minipage}{0.475\textwidth}
\centering
\tdplotsetmaincoords{70}{120}
\begin{tikzpicture}[tdplot_main_coords,scale=4.45]
\coordinate (O) at (0,0,0);
\draw[->,thin] (O) -- (1.10,0,0) node[anchor=north east,fill=white,inner sep=1pt] {$f_1$};
\draw[->,thin] (O) -- (0,1.10,0) node[anchor=west,fill=white,inner sep=1pt] {$f_2$};
\draw[->,thin] (O) -- (0,0,1.10) node[anchor=south,fill=white,inner sep=1pt] {$f_3$};
\coordinate (E1) at (1,0,0);
\coordinate (E2) at (0,1,0);
\coordinate (E3) at (0,0,1);
\filldraw[fill=softgray,draw=black!65,opacity=0.35] (E1) -- (E2) -- (E3) -- cycle;
\draw[black!70,thick] (E1) -- (E2) -- (E3) -- cycle;
\foreach \P in {(0,0,1),(0,0.333333,0.666667),(0,0.666667,0.333333),(0,1,0),(0.333333,0,0.666667),(0.333333,0.666667,0),(0.666667,0,0.333333),(0.666667,0.333333,0),(1,0,0)}{
  \filldraw[magred] \P circle (0.66pt);
}
\node[magred,anchor=north,fill=white,inner sep=1pt] at (0.48,-0.07,0) {$9$-pt Mag};
\end{tikzpicture}
\end{minipage}
\caption{Exact $9$-point 3D simplex runs from the Das--Dennis initialization with the centroid removed. Hypervolume in blue converges to the three vertices plus the six-point symmetric HV orbit; magnitude in red remains at the Das--Dennis pattern.}
\label{fig:3d9}
\end{figure}

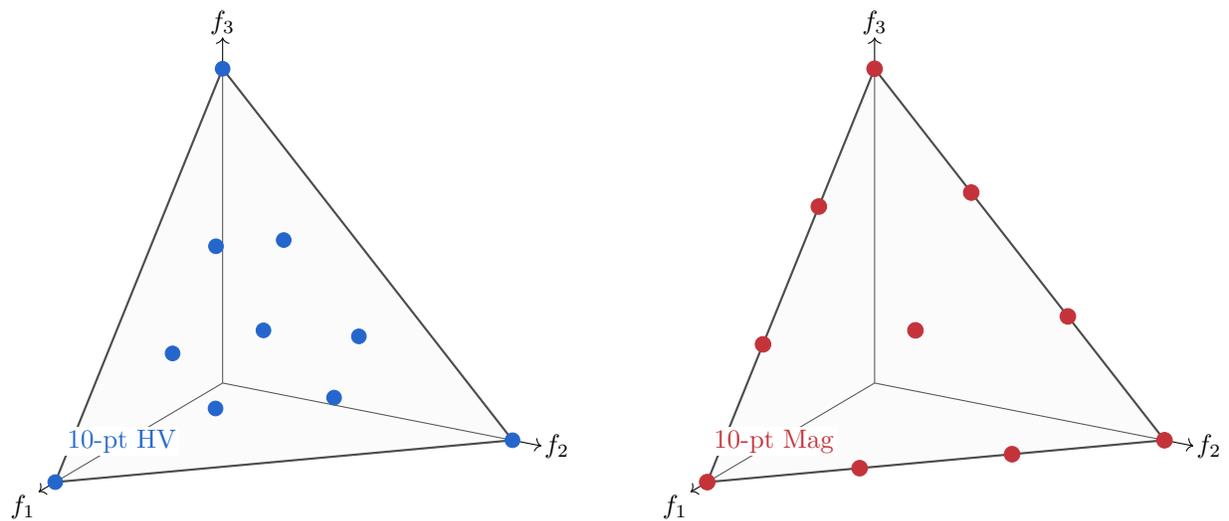
\begin{figure}[p]
\centering
\begin{minipage}{0.475\textwidth}
\centering
\tdplotsetmaincoords{70}{120}
\begin{tikzpicture}[tdplot_main_coords,scale=4.45]
\coordinate (O) at (0,0,0);
\draw[->,thin] (O) -- (1.10,0,0) node[anchor=north east,fill=white,inner sep=1pt] {$f_1$};
\draw[->,thin] (O) -- (0,1.10,0) node[anchor=west,fill=white,inner sep=1pt] {$f_2$};
\draw[->,thin] (O) -- (0,0,1.10) node[anchor=south,fill=white,inner sep=1pt] {$f_3$};
\coordinate (E1) at (1,0,0);
\coordinate (E2) at (0,1,0);
\coordinate (E3) at (0,0,1);
\filldraw[fill=softgray,draw=black!65,opacity=0.35] (E1) -- (E2) -- (E3) -- cycle;
\draw[black!70,thick] (E1) -- (E2) -- (E3) -- cycle;
\foreach \P in {(1,0,0),(0,1,0),(0,0,1),(0.333333,0.333333,0.333333),
                (0.555556,0.296296,0.148148),(0.555556,0.148148,0.296296),
                (0.296296,0.555556,0.148148),(0.296296,0.148148,0.555556),
                (0.148148,0.555556,0.296296),(0.148148,0.296296,0.555556)}{
  \filldraw[hvblue] \P circle (0.62pt);
}
\node[hvblue,anchor=north,fill=white,inner sep=1pt] at (0.48,-0.07,0) {$10$-pt HV};
\end{tikzpicture}
\end{minipage}\hfill
\begin{minipage}{0.475\textwidth}
\centering
\tdplotsetmaincoords{70}{120}
\begin{tikzpicture}[tdplot_main_coords,scale=4.45]
\coordinate (O) at (0,0,0);
\draw[->,thin] (O) -- (1.10,0,0) node[anchor=north east,fill=white,inner sep=1pt] {$f_1$};
\draw[->,thin] (O) -- (0,1.10,0) node[anchor=west,fill=white,inner sep=1pt] {$f_2$};
\draw[->,thin] (O) -- (0,0,1.10) node[anchor=south,fill=white,inner sep=1pt] {$f_3$};
\coordinate (E1) at (1,0,0);
\coordinate (E2) at (0,1,0);
\coordinate (E3) at (0,0,1);
\filldraw[fill=softgray,draw=black!65,opacity=0.35] (E1) -- (E2) -- (E3) -- cycle;
\draw[black!70,thick] (E1) -- (E2) -- (E3) -- cycle;
\foreach \P in {(0,0,1),(0,0.333333,0.666667),(0,0.666667,0.333333),(0,1,0),
                (0.333333,0,0.666667),(0.333333,0.333333,0.333333),(0.333333,0.666667,0),
                (0.666667,0,0.333333),(0.666667,0.333333,0),(1,0,0)}{
  \filldraw[magred] \P circle (0.66pt);
}
\node[magred,anchor=north,fill=white,inner sep=1pt] at (0.48,-0.07,0) {$10$-pt Mag};
\end{tikzpicture}
\end{minipage}
\caption{Exact $10$-point 3D simplex runs from the full level-$3$ Das--Dennis grid. Hypervolume in blue converges to the three vertices, the centroid, and a symmetric six-point orbit. Magnitude in red remains at the complete Das--Dennis grid.}
\label{fig:3d10}
\end{figure}

\section{How to compute the magnitude}

So far we computed magnitude and hypervolume by inclusion--exclusion formulas. This is convenient for analysis, but computationally very inefficient. Therefore we next discuss how magnitude can be computed more efficiently.

Let
\[
A=\{a^{(1)},\dots,a^{(n)}\}\subseteq [0,\infty)^m,
\qquad n:=|A|,
\]
be a finite approximation set for a multiobjective maximization problem, and let
\[
D(A):=\bigcup_{i=1}^{n}[0,a^{(i)}]
\]
be its dominated set anchored at the origin. By the projection formula established above,
\[
\Mag(D(A))
=
1+\sum_{\emptyset\neq S\subseteq [m]}2^{-|S|}\,
\lambda_{|S|}\!\bigl(\pi_S D(A)\bigr),
\]
where \([m]=\{1,\dots,m\}\), \(\pi_S\) denotes projection onto the coordinate subset \(S\), and \(\lambda_{|S|}\) denotes \(|S|\)-dimensional Lebesgue measure.

This formula immediately suggests a practical computational strategy: for every nonempty coordinate subset \(S\subseteq [m]\), one computes the measure of the projected dominated set \(\pi_S D(A)\), and then combines the resulting values with the weights \(2^{-|S|}\). The next proposition shows that each projected term is itself a hypervolume of a projected dominated set.

\begin{proposition}[Projection-based computation of magnitude]
Let \(A=\{a^{(1)},\dots,a^{(n)}\}\subseteq [0,\infty)^m\) and let
\[
D(A):=\bigcup_{i=1}^{n}[0,a^{(i)}]
\]
be its dominated set anchored at the origin. Then, for every nonempty \(S\subseteq [m]\),
\[
\pi_S D(A)=D(\pi_S A),
\]
where \(D(\pi_S A)\) is the dominated set generated by the projected approximation set
\(\pi_S A\subseteq [0,\infty)^{|S|}\). Consequently,
\[
\Mag(D(A))
=
1+\sum_{\emptyset\neq S\subseteq [m]}2^{-|S|}\,
\HV_{|S|}\!\bigl(D(\pi_S A)\bigr),
\]
where \(\HV_{|S|}\) denotes \(|S|\)-dimensional hypervolume.
\end{proposition}

\begin{proof}
Since
\[
D(A)=\bigcup_{a\in A}[0,a],
\]
we have, for every nonempty \(S\subseteq [m]\),
\[
\pi_S D(A)
=
\bigcup_{a\in A}\pi_S([0,a])
=
\bigcup_{a\in A}[0,\pi_S a]
=
D(\pi_S A).
\]
Hence
\[
\lambda_{|S|}\!\bigl(\pi_S D(A)\bigr)
=
\HV_{|S|}\!\bigl(D(\pi_S A)\bigr),
\]
and substitution into the projection formula yields the claim.
\end{proof}

The preceding proposition shows that magnitude can be computed entirely in terms of hypervolume computations on projected approximation sets. This gives an immediate general complexity bound.

\begin{proposition}[Complexity bound via projected hypervolumes]
Let \(T_{\HV}(k,n)\) denote the complexity of computing the \(k\)-dimensional hypervolume of a dominated set generated by \(n\) points. Then
\[
T_{\Mag}(m,n)\;\le\;\sum_{k=1}^{m}\binom{m}{k}\,T_{\HV}(k,n).
\]
In particular, if lower-dimensional hypervolume computations are asymptotically no harder than the \(m\)-dimensional one, then for fixed \(m\),
\[
T_{\Mag}(m,n)=O\!\bigl((2^m-1)\,T_{\HV}(m,n)\bigr).
\]
\end{proposition}

\begin{proof}
There are exactly \(\binom{m}{k}\) coordinate subsets \(S\subseteq [m]\) of size \(k\). By the previous proposition, each summand in the magnitude formula is obtained from the hypervolume of the projected dominated set \(D(\pi_S A)\). Summing over all nonempty \(S\subseteq [m]\) gives the stated bound.
\end{proof}

In dimensions two and three, this reduction is especially effective, because only a constant number of projected dominated sets have to be evaluated. Combining this observation with the sharp hypervolume complexity bounds of Beume et al.~\cite{BeumeEtAl2009} yields the following corollary.

\begin{corollary}[Sharp complexity in two and three dimensions]
For \(m=2\) and \(m=3\), the magnitude of a dominated set generated by \(n\) points can be computed in
\[
\Theta(n\log n).
\]
That is,
\[
T_{\Mag}(2,n)=\Theta(n\log n),
\qquad
T_{\Mag}(3,n)=\Theta(n\log n).
\]
\end{corollary}

\begin{proof}
For \(m=2\), magnitude requires one two-dimensional hypervolume computation and two one-dimensional projection terms. For \(m=3\), it requires one three-dimensional hypervolume computation, three two-dimensional projection hypervolumes, and three one-dimensional projection terms. Since the number of projected dominated sets is constant for fixed dimension, the asymptotic complexity is the same as that of hypervolume in dimensions two and three. By Beume et al.~\cite{BeumeEtAl2009}, this complexity is \(\Theta(n\log n)\) in both cases.
\end{proof}

For dimensions \(m\ge 4\), the same projection-based reduction still applies, but the complexity is governed by the best available higher-dimensional hypervolume algorithms. Thus
\[
T_{\Mag}(m,n)\;\le\;\sum_{k=1}^{m}\binom{m}{k}\,T_{\HV}(k,n),
\]
and the complexity of magnitude inherits the best currently available upper bounds for hypervolume in each projected dimension. In particular, the higher-dimensional computation of magnitude lies in the same general complexity regime as hypervolume and related instances of Klee's measure problem; see Chan~\cite{Chan2013}.

\begin{remark}
The previous results show that the complexity of computing magnitude differs from that of hypervolume only through the additional lower-dimensional projection terms. Thus, for fixed dimension, magnitude has the same asymptotic complexity as hypervolume up to a constant factor determined by the number of nonempty coordinate projections. In dimensions two and three, this yields the sharp \(\Theta(n\log n)\) bound via Beume et al.~\cite{BeumeEtAl2009}. In higher dimensions, the reduction places the computation of magnitude in the same general complexity regime as hypervolume and related instances of Klee's measure problem; see Chan~\cite{Chan2013}.
\end{remark}

\section{Possible advantages and disadvantages of magnitude as a set-indicator in Pareto Optimization}

The preceding results suggest that magnitude is neither merely a reformulation of hypervolume nor an automatically superior replacement. It is therefore useful to summarize explicitly what may count as an advantage and what may count as a limitation.

Compared with hypervolume, magnitude offers a richer notion of size for dominated sets. Its main potential advantage is that it does not only measure the top-dimensional dominated volume, but also includes positive lower-dimensional projection and boundary contributions. In particular, approximation points that share one or more coordinates with the anchor point may still contribute positively to magnitude, even when their top-dimensional hypervolume contribution is zero. This makes magnitude more sensitive to extreme and boundary structure and may be advantageous in settings where boundary-including approximations are desirable. More fundamentally, if the aim is to measure all relevant geometric information carried by a finite non-dominated approximation set, rather than only its full-dimensional dominated volume, then magnitude has a stronger conceptual motivation: it treats lower-dimensional shadows and boundary contributions as part of the size of the dominated set. The numerical examples in this paper suggest exactly this behavior: compared with hypervolume, magnitude tends to favor boundary-supported populations and, for suitable cardinalities, regular Das--Dennis-type configurations.

At the same time, this feature can also be viewed as a disadvantage. If the main objective is to reward only full-dimensional dominated volume, then the additional lower-dimensional terms may introduce a boundary bias that is not wanted. Moreover, these extra terms entail an additional computational burden, since magnitude requires the evaluation of lower-dimensional projected dominated sets in addition to the top-dimensional one. Hypervolume also has a direct probabilistic interpretation: for a suitable random goal vector in objective space, its value equals the probability that at least one point of the approximation set weakly attains that goal vector in all objectives; see Emmerich, Deutz, and Yevseyeva~\cite{EmmerichDeutzYevseyeva2014}. For magnitude, such a probabilistic interpretation is no longer immediate and would at least have to be revisited, precisely because lower-dimensional shadow and boundary contributions enter in addition to the full-dimensional dominated volume. Thus, whether magnitude should be preferred over hypervolume depends on whether one regards lower-dimensional boundary contributions as relevant information or as an undesirable distortion of pure volume-based quality assessment, and on whether the probabilistic interpretation of hypervolume is considered essential in the application at hand.

\section{Conclusion}

The results obtained here identify magnitude as a mathematically natural set indicator for finite Pareto front approximations. For origin-anchored dominated sets generated by finite approximation sets, we prove strict Pareto compliance via strict set monotonicity and show that magnitude is induced by a notion of size that records not only the full-dimensional dominated volume but also all lower-dimensional projection contributions. This leads to indicator behavior that differs systematically from hypervolume. In particular, the three-dimensional simplex examples suggest that complete Das--Dennis grids occupy a distinguished position for magnitude, and the observed convergence to such grids for suitable cardinalities points to a structural phenomenon that deserves a theoretical explanation. The projected gradient-based computations reported here provide a first indication of the corresponding magnitude-optimal \(\mu\)-distributions. At the same time, the projection formula implies that, for fixed dimension, the complexity of computing magnitude remains in the same asymptotic regime as that of hypervolume, differing only by the constant factor corresponding to the \(2^d-1\) nonempty coordinate projections.

This perspective places magnitude naturally within a broader, mathematically grounded theory of quality indicators for finite Pareto front approximations.

The next steps are naturally organized around several complementary topics. First, indicator-optimal $\mu$-distributions for magnitude should be studied on structured fronts, in analogy with the hypervolume literature on optimal distributions and reference-point effects~\cite{AugerBaderBrockhoff2010,AugerBaderBrockhoffZitzler2009,Brockhoff2010,IshibuchiImadaMasuyamaNojima2019,IshibuchiNanPang2025,Singh2025}. Second, one should develop dedicated methods for maximizing magnitude, including indicator-based evolutionary algorithms and direct set-based local methods such as projected-gradient, Newton-type, and constrained variants~\cite{EmmerichBeumeNaujoks2005,EmmerichDeutz2014,HernandezSchuetzeWangDeutzEmmerich2018,SosaHernandezSchuetzeEmmerich2014,WangDeutzBackEmmerich2017,WangEtAl2023Constraints,ZitzlerKuenzli2004}. Third, the computation of magnitude itself deserves further study, both algorithmically and analytically: efficient evaluation via projections and inclusion--exclusion, sharper complexity bounds in higher dimension, and explicit formulas or efficient routines for gradients and Hessians, including the sparsity structure induced by lower-dimensional terms~\cite{BeumeEtAl2009,Chan2013,DeutzEmmerichWang2023,GuerreiroFonsecaPaquete2021}. Finally, it would be valuable to study the magnitude of finite sets directly, since that viewpoint may be relevant for discrete optimization problems with integer-valued objective functions and for diversity enhancement beyond dominated-box geometry~\cite{Huntsman2023}.

\end{document}